\documentclass{amsart}
\usepackage[utf8]{inputenc}
\usepackage{amsthm}
\usepackage{comment}
\usepackage{appendix}
\usepackage{bbm}
\usepackage{lipsum}
\usepackage{cite}

\usepackage{color}

\makeatletter
\def\thm@space@setup{%
  \thm@preskip=\parskip \thm@postskip=0pt
}
\makeatother

\numberwithin{equation}{section}

\newcommand{\td}{\mathrm{d}}
\newcommand{\tnu}{\tilde{\nu}^{C,R}_{N,\alpha}}
\newcommand{\nnu}{\nu^{R}_{N,\alpha}}
\newcommand{\cv}{\text{Cov}}
\newcommand{\hs}{\hat \sigma}
\newcommand{\hr}{\hat \rho}
\newcommand{\hx}{\hat \eta}

\newtheorem{thm}{Theorem}[section]
\newtheorem{lm}[thm]{Lemma}
\newtheorem{cor}[thm]{Corollary}

\theoremstyle{definition}
\newtheorem{defn}[thm]{Definition}

\newtheorem{ass}[thm]{Assumption}

\begin{document}
\lipsum[0]
\title{Spin Distributions for Generic Spherical Spin Glasses}

 \author{Arka Adhikari}
 \address{Harvard University}
 \email{adhikari@math.harvard.edu}

\begin{abstract}
    This paper investigates p-spin distributions for a generic spherical p-spin model; we give a representation of spin distributions in terms of a stochastic process. In order to do this, we find a novel double limit scheme that allows us to treat the sphere as a product space and perform cavity computations. The decomposition into a product space involves the creation of a renormalized sphere, whose scale $R$ will be taken to infinity.
\end{abstract}
\maketitle
\section{Introduction}

Spin glasses are used as models for understanding the properties of disordered magnets, where, unlike ferromagnetic and paramagnetic models, the directions of the magnetic fields are oriented in arbitrary directions. In 1975, Sherrington and Kirkpatrick proposed a mean field model in order to gain a foothold on some of the more computational aspects of problem.\cite{SKMod} The SK Hamiltonian is
\begin{equation}
    H_N(\sigma) = \frac{1}{\sqrt{N}} \sum_{i,j=1}^N J_{ij} \sigma_i \sigma_j
\end{equation}
where $\sigma=(\sigma_1,\ldots, \sigma_j) \in \{-1,1\}^N$ ,$J_{ij}$ i.i.d. $N(0,1)$.

The study of the Sherrington-Kirkpatrick model revealed multiple physically novel aspects of spin glasses that are not present in other magnetic models. The Parisi Formula for the free energy \cite{Parisi1, Parisi2} showed that the limiting overlap distribution $\zeta(\cdot \in [a,b])=\lim_{N\rightarrow \infty}\mathbb{P}(\langle \sigma^1,\sigma^2 \rangle \in [a,b])$   where $\sigma^1$ and $\sigma^2$ are two distinct replicas distributed according to the Gibbs' measure from $H_N$ is an important order parameter.

An interpretation of the Parisi formula by M. Mezard, et al. \cite{Mezard1,Mezard2,Mezard3}, led to an interpretation of the Gibbs' measure as being divided into pure states which organize in a hierarchical structure satisfying ultrametricity. The Parisi formula was proved by Talagrand in \cite{TalagrandParisi} while ultrametricity for generic models was established by Panchenko in \cite{PanchenkoUltra}.

Though the study of the free energy gave many insights into the physical nature of the Gibbs' measure, the free energy to highest order by itself is not able to give very fine details about the Gibbs' distribution. In this paper, we would like to consider the question of spin distributions; this means we will look at the behavior of individual coordinates $\sigma_i^l$.

Such information may be gleaned if one has access to higher order corrections to the free energy formula. In general, this is a difficult problem and requires very precise analysis of explicit formulas as in \cite{Subag,SubagZei}. For the question of spin distributions on the hypercube, Panchenko found a series of invariance principles called the cavity equations \cite{PanchenkoSpin} that allows one to compute individual spin distributions under the assumption of finite replica symmetry breaking. For generic models, Auffinger and Jagganath \cite{JaAuSpin} applied these cavity equations for the hypercube and took a limit procedure to study the spin distributions of models that satisfy full replica symmetry breaking. They found a representation of the p-spin distributions on the hypercube in terms of a stochastic process defined on the support of the limiting Gibbs' measure.

In this paper, we investigate the asymptotic spin distributions of a generic Hamiltonian on the sphere $S_N(N)$ where $S_k(r)$ denotes the k-dimensional sphere of radius $\sqrt{r}$.
The Hamiltonian is defined via its covariance structure. For two points $\sigma=(\sigma_1,\sigma_2,\ldots,\sigma_N)$ and $\rho=(\rho_1,\rho_2,\ldots,\rho_N)$ we have the following Hamiltonian $H_N$
\begin{equation}\label{eq:defHam}
    \text{Cov}(H_N(\sigma),H_N(\rho)) = N \xi(R_{12})
\end{equation}
where the overlap $R_{12}$ between $\sigma$ and $\rho$ is given by
$$
R_{12} = \frac{1}{N} \sum_{i=1}^N \sigma_i \rho_i
$$

If we want to write out our Hamiltonian in terms of its p-spin components, we would have
\begin{equation}\label{eq:defpspinHam}
    \sum_{p=2}^{\infty} \beta_p \frac{1}{N^{(p-1)/2}}\sum_{i_1,\ldots,i_p} g_{i_1,\ldots,i_p}\sigma_{i_1}\sigma_{i_2}\ldots \sigma_{i_p}
\end{equation}
We would then see that the function $\xi$ appearing in \eqref{eq:defHam} can be written as
\begin{equation}
    \xi(x) = \sum_{p=2}^{\infty} \beta_p^2 x^p
\end{equation}

We have the following assumptions on the Hamiltonian $H_N$
\begin{ass}\label{asmp:Ham}

\begin{itemize}
        \item We have the Hamiltonian $H_N$ is a generic model. This allows us to apply Panchenko's Ultrametricity Theorem \cite{Panchenkobook}[Col 3.2,Thm 2.14]and establish uniqueness of the overlap distribution.
    \item We have sufficient decay on the coefficients $\beta_p$; namely, we will assume that $\beta_p \le \frac{1}{p^2}$. This will later allow us to perform perturbation theory carefully.
\end{itemize}
\end{ass}
For the spherical model, we no longer have the cavity equation, so we require a new  invariance principle that will allow us to get a representation of the spin distributions.

\subsection{Main Results}

We quickly review the necessary definitions from the theory of spin distributions as from \cite{Panchenkobook}[Chapter4]. Let $G_N$ be the Gibbs' measure corresponding to the Hamiltonian $H_N$ and let $(\sigma^l)$ be a series of replicas i.i.d. distributed according to $G_N$.  Consider the spin matrix $(S^N)=(\sigma^l_i)_{1\le i\le N, 1 \le l}$ and the overlap matrix $[R^N]_{l,l'} = \langle\sigma^l, \sigma^{l'}\rangle$; we extend $S^N$  to be an $\infty \times \infty$ matrix by setting $\sigma^l_j =0$ for $j >N$. We let $\mu_N$ be the law of $S^N$ and $\eta_N$ be the distribution of the overlap matrix $[R^N]_{l,l'}$  distributed according to $\mathbb{E}(G_N^{\otimes \infty})$. We will say that $\mu$( respectively $\eta$ )is a limiting spin (respectively overlap) distribution if it is a subsequential limit of some $\mu_N$ (respectively $\eta_N$)  in the sense of convergence of finite dimensional distributions.

Notice that $\eta$ is a weakly exchangeable distribution of positive semidefinite matrices. We can use the Dobvysh-Sudakov theorem \cite{DobSud} in order to say there is a random measure $\nu$ on $L^2[0,1] \times \mathbb{R}$ satisfying the following property. The distribution of the overlap matrix $(R_{l,l'})$ from $\eta$ is the same as the distribution of $(a_l \delta_{l,l'} + <h_l,h_l'>)$ where $(a_l,h_l)$ are sampled i.i.d from $\nu$ with respect to the randomness of $\nu$. The measure $\nu$ restricted to the space $L^2[0,1]$ is called an asymptotic Gibbs' measure for the spin glass.

For generic p-spin models, we can say more about the distribution $\eta$. Because of the Ghirlanda-Guerra Identities, it is known that the law of the distribution of the entire overlap matrix is a function of the law of the overlap distribution $R_{1,2}$ \cite{Panchenkobook}[Thm 2.13]. In addition, for generic p-spin models, the limiting distribution of $R_{1,2}$ is unique.

     We can give a more explicit description of the asymptotic measure $\mu$ in terms of a stochastic process.
     We will use much of the notation from \cite{JaAuSpin}.

          Consider $q_*>0$ and let $U$ be a positive  ultrametric subset of the sphere of radius $\sqrt{q_*}$ in the space $L^2[0,1]$. We define a Brownian motion on $U$; consider the Gaussian process $B_t(\sigma)$ indexed by $(t,\sigma) \in [0,q_*]\times U$ which is centered, a.s. continuous in time and in space with covariance
          \begin{equation} \label{eq:BrownMot}
              \text{Cov}(B_t(\sigma_1), B_t(\sigma_2)) = (t_1 \wedge t_2) \wedge (\sigma_1,\sigma_2)
          \end{equation}

        Using this Brownian motion, we can define the cavity field $Y_t(\sigma)$ on $U$ as the solution of the following SDE
        \begin{equation} \label{def:cavfield}
        \begin{aligned}
            & \td \mathcal{Z}_t(\sigma) = \sqrt{\xi''(t)}\td B_t(\sigma)\\
            & \mathcal{Z}_0(\sigma)=0
        \end{aligned}
        \end{equation}

     Let $\zeta$ be the Parisi measure for the generic p-spin model.
     Consider the following Parisi initial value problem on $(0,1) \times \mathbb{R}$
     \begin{equation}
         \begin{aligned}
             & u_t + \frac{\xi''(t)}{2}(u_{xx} + \zeta([0,t]) u_x^2)=0\\
             & U(1,x) = \frac{x^2}{2(1+S_{\zeta})}
         \end{aligned}
     \end{equation}
     where $S_{\zeta} = \int_{0}^{1} \zeta([0,l])(l \xi''(l) + \xi'(l)) \td l$.

     We can now define the local field $X_t(\sigma)$ on $[0,q_*] \times U$ by an SDE
     \begin{equation} \label{def:locfield}
     \begin{aligned}
         & \td \mathcal{X}_t(\sigma) = \xi''(t) \zeta([0,t]) u_x(t,\mathcal{X}_t(\sigma))\td t + \td \mathcal{Z}_t(\sigma) \\
         & \mathcal{X}_0(\sigma)=0
    \end{aligned}
     \end{equation}

     Many basic properties of these stochastic processes are covered in section A.1 of \cite{JaAuSpin}.
     With these preliminaries established, we can use these stochastic processes to create probability distributions for spins.

     We define the measure $\rho^f_{\zeta}$ where $\sigma \in U$ and $f \in \mathbb{R}$ as follows. Let $g$ be a bounded function on $\mathbb{R}$ with compact support.
     \begin{equation} \label{eq:defrho}
         \rho^f_{\zeta}(g) = \frac{\int_{-\infty}^{\infty}\int_{-\infty}^{\infty}g(s) e^{-\frac{[1+S_\zeta]}{2} s^2} e^{-\frac{(y-f)^2}{2[\xi'(1) - \xi'(q)]}} e^{sy} \td y \td s}{\int_{-\infty}^{\infty}\int_{-\infty}^{\infty} e^{-\frac{[1+S_\zeta]}{2} s^2} e^{-\frac{(y-f)^2}{2[\xi'(1) - \xi'(q)]}} e^{sy} \td y \td s}
     \end{equation}

     We will establish the following theorem
     \begin{thm} \label{them:mainthm}
     Consider a spherical Hamiltonian $H_N$ satisfying the assumptions of Assumption \ref{asmp:Ham}. Let $\zeta^*(\cdot)= \lim_{N \rightarrow \infty} \mathbb{E}\langle R_{1,2} \in \cdot\rangle$ be the limiting overlap distribution and let $q_*= \sup\text{supp}(\zeta^*)$. Let $\mu$ be an asymptotic spin distribution and $\nu$ be an asymptotic Gibbs' measure corresponding to $H_N$.

     We define the following measure on $\infty \times  \infty$ matrices: choose $\{\sigma^l\}$ i.i.d from $\nu^{\otimes \infty}$. Independently construct $\mathcal{X}^i_t(\sigma)$ distributed according to \eqref{def:locfield} on $U = \text{supp}(\nu)$. Distribute $S_i^l$ according to the measure $\rho^{\mathcal{X}^i_{q_*}(\sigma^l)}_{\zeta^*}$.

     Let $(s_i^l)$ be spins distributed according to $\mu$. Then $(s^i_l)$ and $(S^i_l)$ are equal in distribution.
     \end{thm}

 We will give a brief overview of the main strategy of this paper. As in the paper \cite{JaAuSpin} for the hypercube, one has a set of invariance principles called the cavity equations \cite{PanchenkoSpin} that gives a useful representation of spin distributions. It involves treating the $N$ dimensional hypercube as a product of the $N-n$ and $n$ dimensional hypercubes. For $n \ll N$, we can split the $N$ dimensional Hamiltonian as a Hamiltonian on $N-n$ dimensional space and an independent perturbation Hamiltonian for the last $n$ particles. Due to this independence, one can explicitly compute the spin distribution for the last $n$ particles and show that it is a function of the overlap distribution of the spin glass.

 However, the sphere is not naturally represented as a product space and it is not immediately clear how one can apply a cavity computation. Our main innovation here is to find a limiting scheme by which we can interpret the sphere $S_N(N)$ as a product set of a sphere of smaller dimension and to find an appropriate splitting of the Hamiltonian into a main term and preturbative part.

To define the splitting, we first fix a renormalization scale $R$ for the first $N-1$ particles, $\hat \sigma_1^2+ \ldots + \hat \sigma_{N-1}^2 = N -R^2$ with the map that $\sigma_1 = \frac{\sqrt{N- \sigma_N^2}}{\sqrt{N-R^2}}  \hat \sigma_1$. This way the Hamiltonian for the sphere on $N$-particles can be treated as a Hamiltonian on the product of a sphere of $N-1$ particles on renormalized scale $N-R^2$ and $\mathbb{R}$ for the final particle.

When treating the sphere as a product set as we have described earlier, the Hamiltonian decomposes into three parts. The first is the standard spherical Hamiltonian on $N-1$ particles. The second is a local field term using those terms in $H_N$ that explicitly used the spin $\sigma_N$. The third is a novel renormalization term that represents the rescaling of energy levels due to the term $ \frac{\sqrt{N- \sigma_N^2}}{\sqrt{N-R^2}}$ on the $N-1$ sphere Hamiltonian.

This new renormalization term and the Hamiltonian on $N-1$ particles are coupled to each other. For the purposes of computation, we approximate this new renormalization term as a Gaussian process independent of the base Hamiltonian on $N-1$ particles. What we observe is that this error from treating this term as an independent perturbation becomes smaller as the renormalization radius $R \rightarrow \infty$. Ultimately, the main effect of this renormalization is to change the variance of the Gaussian spin of the last particle on the sphere.

Once we have established the cavity equations in this way, we are able to represent the spin distributions in terms of our stochastic process as in Theorem \ref{them:mainthm}. We first show this equality under the replica symmetry breaking assumption and show the result for all limiting overlap distributions by taking an appropriate limit.

\section*{Acknowledgments}
The author is grateful to Aukosh Jagannath for introducing him to this problem and for useful discussions.
\section{The Cavity Equation for the Sphere}

As we have described in the introduction, the first step to proving the Theorem \ref{them:mainthm} is to first prove a version of the cavity equation. To this end, we need to create an auxiliary space so that the auxiliary space is represented as a product space and so we have a natural notion of a cavity equation.
\subsection{Construction of the auxiliary Hamiltonian $\tilde{H}_N$}
Our first step in constructing the auxiliary space is to restrict the sphere so that the final $n$ coordinates lie in a compact interval.
\begin{equation}
    S_{N+n}^{C,n} = \{(\sigma_1,\ldots, \sigma_N,\sigma_{N+1},\ldots,\sigma_{N+n}): \sigma_1^2+ \ldots + \sigma_N^2 + \sigma_{N+1}^2 + \ldots \sigma_{N+n}^2 = N+n, |\sigma_{N+1}|,\ldots |\sigma_{N+n}| \le C )\}
\end{equation}
with Hamiltonian $H_{N+n}$.

We begin with the following remark
\begin{lm} \label{lm:Creplace}
\begin{equation}
    \lim_{N \rightarrow \infty}\mathbb{E}\left[\prod_{l=1}^{k}\langle (\sigma^{l}_{N+1})^{e_{l,1}} \ldots (\sigma^{l}_{N+n})^{e_{l,n}}\rangle_{H_{N+n}(S_{N+n}(N+n))} \right] = \lim_{C \rightarrow \infty} \lim_{N\rightarrow \infty}\mathbb{E}\left[\prod_{l=1}^{k} \langle (\sigma_{N+1}^l)^{e_{l,1}}\ldots (\sigma_{N+n}^l)^{e_{l,n}}\rangle_{H_{N+n}(S_{N+n}^{C,n})}\right]
\end{equation}
where $H_{N+n}(A)$ designates the Gibbs' average of the Hamiltonian $H_{N+n}$ restricted to the set $A$.
Here,the $l$ is a superscript designating the replica index while $e_{l,n}$ is the exponent computed.
\end{lm}
\begin{proof}
Each of the spins $\sigma_{N+1},\ldots,\sigma_{N+n}$ have subgaussian tails. The contribution to the Gibbs' average from spins that lie outside the set $S_{N+n}^{C,n}$ becomes increasingly small as $C$ gets large.
\end{proof}
 In order to construct the auxiliary Hamiltonian, we would like to consider $S_{N+n}^{C,n}$ as a product set.
\begin{equation}
    S_{N+n}^{C,R,n}= \{(\hat \sigma_1,\ldots, \hat \sigma_N): \hat \sigma_1^2 + \ldots + \hat \sigma_N^2 = N - nR^2\} \times \{|\hat \sigma_{N+1}|,\ldots,|\hat \sigma_{N+n}| \le C\}
\end{equation}

    The map between $S_{N+n}^{C,n}$ and $S_{N+n}^{C,R,n}$ is as follows:
    \begin{equation}
        \hat \sigma_i := \frac{\sigma_i \sqrt{N-nR^2}}{\sqrt{N+n - \sigma_{N+1}^2-\ldots -\sigma_N^2}},  1\le i \le N
    \end{equation}
    and $\hat \sigma_{N+i} := \sigma_{N+i}$ for $1 \le i \le n$.
    We will denote the factor $\frac{\sqrt{N+n - \sigma_{N+1}^2- \ldots - \sigma_N^2}}{\sqrt{N-nR^2}}$ as $S$.
    Notice that the first part of the decomposition of $S_{N+n}^{C,R,n}$ is the sphere $S_N(N-nR^2)$

    When defined in terms of the new variables $\hat \sigma_i$, we will be able to write the Hamiltonian $H_{N+n}$ on the set $S_{N+n}^{C,R,n}$ as
    \begin{equation}
        H_{N+n}(\hat \sigma_1,\ldots, \hat \sigma_{N+n})=\sum_{p=2}^{\infty} \beta_p \frac{1}{(N+n)^{(p-1)/2}} \sum_{k=0}^{n} \sum_{ i_1,\ldots i_p:|\{i_1,\ldots,i_p\}\cap\{N+1,\ldots, N+n\}|=k} g_{i_1,i_2,\ldots,i_p} S^{p-k} \hat \sigma_{i_1}\ldots \hat \sigma_{i_p}
    \end{equation}

    Ideally, we would like to write the above Hamiltonian as a Hamiltonian on the sphere $S_N(N-R^2)$ plus a perturbation term. Namely, we have
    \begin{equation}
    \begin{aligned}
        & H_{N+n}(\hat \sigma_1,\ldots, \hat \sigma_{N+n}) = \sum_{p=2}^{\infty} \beta_p \frac{1}{(N+n)^{(p-1)/2}}  \sum_{1 \le i_1\ldots i_p \le N}g_{i_1,i_2,\ldots,i_p}\hat \sigma_{i_1}\ldots \hat \sigma_{i_p}
        \\
        &+ \sum_{p=2}^{\infty} \beta_p \frac{1}{(N+n)^{(p-1)/2}}  \sum_{1\le i_1\ldots i_p \le N }g_{i_1,i_2,\ldots,i_p}\hat \sigma_{i_1}\ldots \hat \sigma_{i_p}(S^p -1)
        \\&+ \sum_{p=2}^{\infty} \beta_p \frac{1}{(N+n)^{(p-1)/2}} \sum_{i_1\ldots i_p:|\{i_1,\ldots,i_p\}\cap\{N+1,\ldots, N+n\}|=1}g_{i_1,i_2,\ldots,i_p}\hat \sigma_{i_1}\ldots \hat \sigma_{i_p}(S^{p-1})\\
        &+ \sum_{p=2}^{\infty} \beta_p \frac{1}{(N+n)^{(p-1)/2}} \sum_{k=2}^n \sum_{i_1\ldots i_p:|\{i_1,\ldots,i_p\}\cap\{N+1,\ldots, N+n\}|=k}g_{i_1,i_2,\ldots,i_p}\hat \sigma_{i_1}\ldots \hat \sigma_{i_p}(S^{p-k})
    \end{aligned}
    \end{equation}

    We make some remarks about the above decomposition; the term on the first line is of order $N$ and is easily seen to be a Hamiltonian defined on the sphere $S_N(N-R^2)$ with coordinates $\hat \sigma_1,\ldots, \hat \sigma_N$. The term on the second line will be a term of order 1, but is coupled to the term on the first line.

    The terms on the third line are additionally of order 1 and independent of the terms on the first two; notice that each of these terms will only involve exactly 1 of the cavity coordinates $\hat \sigma_1,\ldots, \hat \sigma_N$.

    The terms on the fourth line are independent of those that have come before and have variance of order $1/N$. We will be able to treat these terms as errors.

    The issues that arise in computing distributional equivalences come from the second term in the above expression; it is coupled with the main Hamiltonian on the sphere $S_N(N -R^2)$. In order to perform computations in the future,what we would like to do is instead replace this term with one that is independent of the main Hamiltonian on the first line. Namely, we would like to consider
    \begin{equation}\label{eq:AuxHam}
        \begin{aligned}
        &\tilde H_{N+n}(\hat \sigma_1,\ldots, \hat \sigma_{N+n}) = \sum_{p=2}^{\infty} \beta_p \frac{1}{(N+n)^{(p-1)/2}}  \sum_{1\le i_1\ldots i_p\le N}g_{i_1,i_2,\ldots,i_p}\hat \sigma_{i_1}\ldots \hat \sigma_{i_p}
        \\ &+ \sum_{p=2}^{\infty} \beta_p \frac{1}{(N+n)^{(p-1)/2}}
        \sum_{1 \le i_1\ldots i_{p-1} \le N} \sum_{i=1}^n g^{N+i}_{i_1,i_2,\ldots,i_{p-1}} \frac{\sqrt{R^2+1} - \frac{\hat\sigma_{N+i}^2}{2\sqrt{R^2+1}}}{\sqrt{N+n}} \hat \sigma_{i_1}\ldots \hat \sigma_{i_{p-1}} \\
        & + \sum_{p=2}^{\infty} \beta_p \frac{1}{(N+n)^{(p-1)/2}} \sum_{i_1\ldots i_p:|\{i_1,\ldots,i_p\}\cap\{N+1,\ldots, N+n\}|=1}g_{i_1,i_2,\ldots,i_p}\hat \sigma_{i_1}\ldots \hat \sigma_{i_p}
        \end{aligned}
    \end{equation}
    where the $g^{N+i}_{i_1,\ldots,i_{p-1}}$ are $N(0,1)$ Gaussian random variables independent of all other randomness. Notice also that in the local field term, we have dropped the rescaling term $S^{p-1}$ as we will eventually be able to show, provided sufficient decay of the $\beta_p$, that $S^{p-1} -1$ will lead to a term that is of smaller order. Along the same line, we have removed the fourth line of error terms, anticipating that they will eventually be shown to be insignificant.

    For simplicity of notation later, we will denote
    \begin{equation} \label{Eq:hath}
        \hat H_N = \sum_{p=2}^{\infty} \beta_p \frac{1}{(N+n)^{(p-1)/2}}  \sum_{1\le i_1\ldots i_p\le N}g_{i_1,i_2,\ldots,i_p}\hat \sigma_{i_1}\ldots \hat \sigma_{i_p}
    \end{equation}
    \begin{equation}\label{eq:Y}
        Y^i(\hat \sigma):= \sum_{p=2}^{\infty} \beta_p \frac{1}{(N+n)^{p/2}} \sum_{1\le i_1,\ldots,i_{p-1} \le N} g^{N+i}_{i_1,\ldots,i_{p-1}} \hat \sigma_{i_1} \ldots \hat \sigma_{i_{p-1}}
    \end{equation}
    \begin{equation}\label{eq:Z}
        Z^i(\hat \sigma):= \sum_{p=2}^{\infty} \beta_p \frac{1}{(N+n)^{p/2}} \sum_{i_1\ldots i_p:\{i_1,\ldots,i_p\} \cap\{N+1,\ldots, N+n\}=\{N+i\}} g_{i_1,\ldots,i_p} \prod_{k:i_k \ne N+i}\hat{\sigma}_{i_k}
    \end{equation}
     This is a manifest shorthand for the three lines in the decomposition of \eqref{eq:AuxHam}

    In order to justify the replacement of the Gaussian process $H_{N+n}$ with $\tilde H_{N+n}$, we compare the variances for the two processes.

    We have the following equation provided one assumes sufficient decay of the $\beta_p$ terms
    \begin{equation}
    \begin{aligned}
        & \text{Cov}(H_{N+n}(\hat \sigma), H_{N+n}(\hat \rho))= (N+n) \sum_{p=2}^{\infty} \beta_p^2 R_{\hat \sigma,\hat\rho}^p + \sum_{p=2}^{\infty} \frac{p}{2}\beta_p^2 (nR^2 +n - \hat \sigma_{N+1}^2 - \ldots - \hat \sigma_{N+n}^2) R_{\hat \sigma,\hat \rho}^p \\
        &+ \sum_{p=2}^{\infty}\frac{p}{2} \beta_p^2 (nR^2 +n - \hat \rho_{N+1}^2 - \ldots - \hat \rho_{N+n}^2) R_{\hat \sigma,\hat \rho}^p + \sum_{p=2}^{\infty} p \beta_p^2 R_{\hat \sigma,\hat \rho}^{p-1} \left(\sum_{i=1}^{n} \hat \sigma_{N+i} \hat \rho_{N+i}\right) + O\left(\frac{1}{N}\right)
    \end{aligned}
    \end{equation}
    where we define the terms
    $$
    R_{\hat \sigma, \hat \rho} = \frac{1}{N+n}\sum_{i=1}^N \hat \sigma_i \hat \rho_i
    $$

    We have performed the following computation various times when one attempts to compute the above covariance for terms such that $p \le N^{1/3}$
    \begin{equation}
    \begin{aligned}
    & \left(\frac{1}{N+n} \sum_{i=1}^N \sigma_i \rho_i\right)^{p} = \left(\frac{1}{N+n} \sum_{i=1}^N \hat \sigma_i \hat \rho_i\right)^{p} \left(1+ \frac{nR^2 +n -\hat \sigma_{N+1}^2- \ldots \hat \sigma_{N+n}^2}{N-R^2}\right)^{p/2}\left(1+ \frac{nR^2 +n -\hat \rho_{N+1}^2- \ldots \hat \rho_{N+n}^2}{N-R^2}\right)^{p/2}\\
    &= R_{\hat \sigma,\hat \rho}^p  + \frac{p}{2} R_{\hat \sigma,\hat \rho}^p \frac{nR^2 +n -\hat \sigma_{N+1}^2- \ldots \hat \sigma_{N+n}^2}{N+n} + \frac{p}{2} R_{\hat \sigma,\hat \rho}^p \frac{nR^2 +n -\hat \rho_{N+1}^2- \ldots \hat \rho_{N+n}^2}{N+n} + O\left(\frac{1}{N}\right)
    \end{aligned}
    \end{equation}

    What we have used is the fact that if $x$ is an order $\frac{1}{N}$ quantity and $p \le \sqrt{N}$, then we have that
    \begin{equation*}
        (1+x)^p \le 1+ px + 2 p^2 x^2 \le 1 +px + O\left(\frac{1}{N^{4/3}}\right)
    \end{equation*}

    We do not attempt to expand the power $\left(1+ \frac{R^2 +n -\hat \rho_{N+1}^2- \ldots \hat \rho_{N+n}^2}{N-R^2}\right)^{p/2}$ when $p \ge N^{1/3}$. Instead, we apply the trivial bound  on the overlap of $\sigma$ and $\rho$, $\frac{1}{N+n}\sum_{i=1}^{N+n} \sigma_i \rho_i \le 1$ and use decay of the temperature terms $\beta_p^2$ in order to show that the contribution from these terms is $O\left(\frac{1}{N}\right)$ in the variance.

    We now compute the covariance of the terms $\tilde H_{N+n}$
    \begin{equation}
    \begin{aligned}
        \text{Cov}(\tilde H_{N+n}(\hat \sigma), \tilde H_{N+n}(\hat \rho))& = (N+n) \sum_{p=2}^{\infty}\beta_p^2 R_{\hat \sigma,\hat \rho}^p + \sum_{p=2}^{\infty} p \beta_p^2 R_{\hat \sigma,\hat \rho}^{p-1}
        \\&+\sum_{p=2}^{\infty}p \beta_p^2 \sum_{i=1}^n \left[\sqrt{R^2 +1}-\frac{\hat \sigma_{N+i}^2}{2 \sqrt{R^2+1}}\right]\left[ \sqrt{R^2 +1}-\frac{\hat \rho_{N+i}^2}{2\sqrt{R^2+1}}\right] R_{\hat\sigma,\hat \rho}^p
    \end{aligned}
    \end{equation}

    We now compute the differences in the respective covariances.
    \begin{equation}\label{eq:CovComp}
    \begin{aligned}
        &\text{Cov}(H_{N+n}(\hat \sigma), \tilde H_{N+n}(\hat \rho)) - \text{Cov}(\tilde H_{N+n}(\hat \sigma), \tilde H_{N+n}(\hat \rho)) =\\
        &\sum_{p=2}^{\infty} \frac{p}{2} \beta_p^2 \sum_{i=1}^n \frac{\hat \sigma_{N+i}^2 \hat \rho_{N+i}^2}{4(R^2+1)} \le \sum_{p=2}^{\infty} \frac{p}{2} \beta_p^2 \frac{n C^4}{4(R^2+1)} = \text{O}\left(\frac{C^4}{R^2}\right)
        \end{aligned}
    \end{equation}

    We see that if we choose $R$ much greater than $C^2$, then the differences in the relative covariance structure will become increasingly small.

    \subsection{Comparison of the Modified Cavity Hamiltonian}

    What we will show now is that the expectation of quantities computed with respect to the Hamiltonian $\tilde{H}_{N+n}$ over the restricted sphere $S_{N+n}^{C,R,n}$ will have small difference from the same quantity computed with respect to the Hamiltonian ${H}_{N+n}$ via an interpolation procedure. We can then proceed to compute spin distributions with respect to the Hamiltonian $\tilde{H}_{N+n}$ that specifically uses the independence of the parts that we have constructed.

    \begin{thm} \label{thm:Rreplacement}
    Let $f$ be a bounded function from $\mathbb{R}^n \rightarrow \mathbb{R}$. Then we have the following comparison estimate
    \begin{equation}\label{eq:Comparison}
    \begin{aligned}
        &\bigg|\mathbb{E}\left[ \frac{\int_{S_{N+n}^{C,R,n}} f(\hat \sigma_{N+1},\ldots,\hat \sigma_{N+n})e^{H_{N+n}(\hat \sigma_1,\ldots,\hat \sigma_{N+n}) } \td \hat \sigma_S}{\int_{S_{N+n}^{C,R,n}} e^{H_{N+n}(\hat \sigma_1,\ldots,\hat \sigma_{N+n}) } \td \hat \sigma_S}\right]
        \\&- \mathbb{E}\left[ \frac{\int_{S_{N+n}^{C,R,n}} f(\hat \sigma_{N+1},\ldots,\hat \sigma_{N+n})e^{\tilde H_{N+n}(\hat \sigma_1,\ldots,\hat \sigma_{N+n}) } \td \hat \sigma_S}{\int_{S_{N+n}^{C,R,n}} e^{\tilde H_{N+n}(\hat \sigma_1,\ldots,\hat \sigma_{N+n}) } \td \hat \sigma_S}\right]\bigg| \le \text{O}\left(\frac{C^4}{R^2}\right)
        \end{aligned}
    \end{equation}
      where the constant in the error bound will be a product of the maximum of the function $f$ and a universal constant. In the integrals that have appeared beforehand, $d\hat \sigma_S$ is the pushforward of the uniform measure on the sphere $\sigma_1^2+ \ldots +\sigma_{N+n}^2 = N+n$ under the map $(\sigma_1,\ldots,\sigma_{N+n}) \rightarrow (\hat \sigma_1,\ldots,\hat \sigma_{N+n}))$
    \end{thm}
    \begin{proof}
    In order to compare the two quantities, we will perform a Gaussian interpolation using the interpolation Hamiltonian where $H_{N+n}$ and $\tilde{H}_{N+n}$ are independent.
    \begin{equation*}
        H_t(\hat \sigma_1,\ldots,\hat \sigma_{N+n}) = \sqrt{t} H_{N+n} + \sqrt{1-t} \tilde H_{N+n}
    \end{equation*}
    We now consider the quantity
    $$
    F(t) = \mathbb{E}\left[\frac{\int_{S_{N+n}^{C,R,n}} f(\hat \sigma_{N+1},\ldots,\hat \sigma_{N+n})e^{H_t(\hat \sigma_1,\ldots,\hat \sigma_{N+n}) } \td \hat \sigma_S}{\int_{S_{N+n}^{C,R,n}} e^{H_t(\hat \sigma_1,\ldots,\hat \sigma_{N+n}) } \td \hat \sigma_S}\right]
    $$
    and remark that the quantity \eqref{eq:Comparison} is $|F(1) - F(0)|$. As is standard, we will derive a bound on the above quantity by bounding the derivative.
    To simplify notation, we will denote the quantity $Z_t:=\int_{S_{N+n}^{C,R,n}} e^{H_t(\hat \sigma_1,\ldots,\hat \sigma_{N+n})} \td \hat \sigma_S$. Since we will always integrate over the set $S_{N+n}^{C,R,n}$ we avoid here any specific mention of this set.
    We have
    \begin{equation}
    \begin{aligned}
        & F'(t) = \mathbb{E}\bigg[ \frac{1}{2\sqrt{t}}(Z_t^{-1} \int f(\hat \sigma) H(\hat \sigma) e^{H_t(\hat \sigma))}\td \hat \sigma_S - Z_t^{-2} \int f(\hat \sigma) H(\hat \rho)  e^{H_t(\hat \sigma) + H_t(\hat \rho)}\td \hat \sigma_S \td \hat \rho_S) - \\
       & \frac{1}{2 \sqrt{1-t}}(Z_t^{-1} \int f(\hat \sigma) \tilde H(\hat \sigma) e^{H_t(\hat \sigma)} \td \hat \sigma_S - Z_t^{-2} \int f(\hat \sigma)\tilde H(\hat \rho) e^{H_t(\hat \sigma) + H_t(\hat \rho)} \td \hat \sigma_S \td \hat \rho_S) \bigg]=\\
       &= \mathbb{E}\bigg[(2Z_t)^{-1}\int f(\hat \sigma) [\cv(H(\hat \sigma),H(\hat \sigma))-\cv(\tilde H(\hs), \tilde H(\hs))] e^{H_t(\hat \sigma)} \td \sigma_S \\
       & - (Z_t)^{-2} \int f(\hat \sigma) [\cv(H(\hat \sigma),H(\hat \rho))-\cv(\tilde H(\hs), \tilde H(\hr))] e^{H_t(\hat \sigma)+ H_t(\hr)}\td \hs_S \td \hr_S\\
       & -(Z_t)^{-2} \int f(\hs) [\cv(H(\hr),H(\hr))-\cv(\tilde H(\hr), \tilde H(\hr))] e^{H_t(\hs) + H_t(\hr)} \td \hs_S \hr_S \\&+ (Z_t)^{-3} \int f(\hs) [\cv(H(\hr),H(\hx))- \cv(\tilde H(\hr), \tilde H(\hx))] e^{H_t(\hs) + H_t(\hr) + H_t(\hx)} \td \hs_S \hr_S \hx_S\bigg]
    \end{aligned}
    \end{equation}
    The first equality merely computed the derivative of the quantities. In order to derive the last expression, we performed an integration by parts.
    On the last line, we apply \eqref{eq:CovComp} and see that the latter quantity will be of order $\text{O}\left(\frac{C^4}{R^2}\right)$. We also remark that the spins $\hat \sigma_i$, will have Gaussian Tails, so this theorem will hold for powers of spins $\hat \sigma_{N+i}^k$
    \end{proof}

    We see that if one takes $R$ much greater than $C^2$, then the distribution with respect to $H$ will be the same as that with respect to $\tilde H$. From this point on, we will now attempt to compute quantities with respect to the distribution using the Hamiltonian $\tilde H$

    Recall the notation \eqref{Eq:hath}, \eqref{eq:Y} and \eqref{eq:Z}.
    Clearly, one can see that
    \begin{equation} \label{Eq:1stCavRep}
    \begin{aligned}
        &\mathbb{E}\prod_{l=1}^k\langle  \ (\hat \sigma_{N+1}^l)^{e_{1,l}}\ldots (\hat \sigma_{N+n}^l)^{e_{l,n}}\rangle_{\tilde H_{N+n}} \\
        &= \mathbb{E}\prod_{l=1}^k \frac{\langle \int_{[-C,C]^n} \prod_{i=1}^n s_i^{e_{l,i}} e^{Z^i(\hat \sigma^l) s_i + \left[\sqrt{R^2 +1} - \frac{s_i^2}{2\sqrt{R^2+1}}\right] Y^i(\hat \sigma^l)}  \left(\frac{N+n - \sum_{i=1}^n s_i^2}{N+n}\right)^{(N+n)/2} \td s_i\rangle_{\hat H_N}}{\langle \int_{[-C,C]^n} \prod_{i=1}^n e^{Z^i(\hat \sigma^l) s_i + \left[\sqrt{R^2 +1} -\frac{s_i^2}{2\sqrt{R^2+1}}\right] Y^i(\hat \sigma^l)} \left(\frac{N+n - \sum_{i=1}^n s_i^2}{N+n}\right)^{(N+n)/2} \td s_i\rangle_{\hat H_N}} \\
        &=  \mathbb{E}\prod_{l=1}^k \frac{\langle \prod_{i=1}^n \int_{[-C,C]}  s_i^{e_{l,i}} e^{-\frac{s_i^2}{2}+ Z^i(\hat \sigma^l) s_i + \left[\sqrt{R^2 +1} -\frac{s_i^2}{2\sqrt{R^2+1}}\right] Y^i(\hat \sigma^l)} \td s_i \rangle_{\hat H_N}}{\langle \prod_{i=1}^n \int_{[-C,C]}  e^{-\frac{s_i^2}{2}+Z^i(\hat \sigma^l) s_i + \left[\sqrt{R^2 +1} -\frac{s_i^2}{2\sqrt{R^2+1}}\right]Y^i(\hat \sigma^l)} \td s_i \rangle_{\hat H_N}} + \text{O}\left(\frac{1}{N}\right)
    \end{aligned}
    \end{equation}

\section{Reduction to Finite Replica Symmetry Breaking}

We will show that due to Ultrametricity,
\begin{equation}\label{eq:impeq}
    \lim_{N \rightarrow \infty} \mathbb{E}\prod_{l=1}^k \frac{\langle \prod_{i=1}^n \int_{[-C,C]}  s_i^{e_{l,i}} e^{-\frac{s_i^2}{2}+ Z^i(\hat \sigma^l) s_i + \left[\sqrt{R^2 +1} -\frac{s_i^2}{2\sqrt{R^2+1}}\right] Y^i(\hat \sigma^l)} \td s_i \rangle_{\hat H_N}}{\langle \prod_{i=1}^n \int_{[-C,C]}  e^{-\frac{s_i^2}{2}+Z^i(\hat \sigma^l) s_i + \left[\sqrt{R^2 +1} -\frac{s_i^2}{2\sqrt{R^2+1}}\right]Y^i(\hat \sigma^l)} \td s_i \rangle_{\hat H_N}}
\end{equation}
only depends on the Hamiltonian $\hat{H}_N$ through its limiting overlap distribution $\zeta^*$, recall this notation from Theorem \ref{them:mainthm}. The right hand side of \eqref{eq:impeq} can be understood as a continuous function $F^{R,C,E}(\zeta)$, where $E$ is the set of all values $e_{l,i}$ evaluated at $\zeta^*$.

Combining the results of \ref{lm:Creplace} , \ref{thm:Rreplacement} and \eqref{Eq:1stCavRep}, we are able to derive the fact that

\begin{equation}
    \lim_{N \rightarrow \infty} \mathbb{E} \prod_{l=1}^k \langle(\sigma_1)^{e_{l,1}} \ldots (\sigma_n^l)^{e_{l,n}}\rangle = \lim_{C \rightarrow \infty} \lim_{R \rightarrow \infty} F^{R,C,E}(\zeta^*)
\end{equation}

What we would like to do is to reduce the computation to when $\zeta^*$ satisfies finite replica symmetry breaking.
Let $\zeta_i$ be a sequence of probability measures approaching $\zeta^*$ in the weak$*$ topology.

We have that
\begin{equation}
    \lim_{N \rightarrow \infty} \mathbb{E} \prod_{l=1}^k \langle(\sigma_1^l)^{e_{l,1}} \ldots (\sigma_n^l)^{e_{l,n}}\rangle= \lim_{C \rightarrow \infty} \lim_{R \rightarrow \infty} \lim_{i \rightarrow \infty}F^{R,C,E}(\zeta_i)
\end{equation}

We would like to exchange the limits so that we can write the limit as
\begin{equation}
    \lim_{i \rightarrow \infty} \lim_{C \rightarrow \infty} \lim_{R \rightarrow \infty} F^{R,C,n}(\zeta_i)
\end{equation}

This would involve showing uniform approach of $F^{R,C,E}(\zeta_i)$ to its limit in $\zeta$. We will proceed to justify this exchange of limits in the following sections.

\subsection{Computation of $F^{R,C,E}(\zeta)$ under Finite Replica Symmetry Breaking}

    \begin{lm}
 There is some function $F^{R,C,E}(\zeta)$ that is continuous on the weak$*$ topology of probability measures on $[0,1]$ such that $\lim_{N \rightarrow \infty} \mathbb{E}\prod_{l=1}^k \langle(\hat \sigma_{N+i}^l)^{e_{l,1}}\ldots (\hat \sigma_{N + n}^l)^{e_{l,n}}\rangle_{\tilde{H}_{N+n}} = F^{R,C,E}(\zeta^*)$ and for measures $\zeta$ satisfying finite replica symmetry breaking with support at points $0 = q_0 \le q_1 \le \ldots \le q_r=q^*$, we have
    \begin{equation} \label{def:fform}
    \begin{aligned}
    &F^{C,R,E}(\zeta) =\\
    &\mathbb{E}\prod_{l=1}^k \frac{\sum_{\alpha^l} w_{\alpha^l} \prod_{i=1}^n \int_{-C}^C s_i^{e_{l,i}} e^{-\frac{1}{2}s_i^2[1+ \xi'(q)(1-q)] + Z^i\left(h_{\alpha^l}\right) s_i + Y^i\left(h_{\alpha^l}\right) \left[\sqrt{R^2 +1} -\frac{s_i^2}{2\sqrt{R^2+1}} \right]  +\frac{(\xi'(1) - q \xi'(q))s_i^4}{4(R^2+1)} } \td s_i}{\sum_{\alpha^l}w_{\alpha^l} \prod_{i=1}^n \int_{-C}^C e^{-\frac{1}{2}s_i^2[1+ \xi'(q)(1-q)] + Z^i\left(h_{\alpha^l}\right) s_i + Y^i\left(h_{\alpha^l}\right) \left[\sqrt{R^2 +1} -\frac{s_i^2}{2\sqrt{R^2+1}} \right]  +\frac{(\xi'(1) - q \xi'(q))s_i^4}{4(R^2+1)} } \td s_i}
    \end{aligned}
    \end{equation}
    where the quantities on the right hand side of the above expression are computed with respect to an RPC whose overlap distribution is given by $\zeta$ and $Y_i$ $Z_i$ are independent Gaussian processes with covariance given by
    \begin{equation} \label{def:MagField}
    \begin{aligned}
        &\text{Cov}(Y_i(h_\alpha), Y_i(h_\beta))=\langle h_\alpha,h_\beta \rangle \xi'(\langle h_\alpha,h_\beta \rangle) \\
        &\text{Cov}(Z_i(h_\alpha),Z_i(h_\beta)) = \xi'(\langle h_\alpha,h_\beta \rangle)
    \end{aligned}
    \end{equation}
    where $l $ is a replica index for $\alpha^l $, $e_{l,i} $ is an exponent, and $ \langle \cdot, \cdot \rangle$ is the inner product.
    \end{lm}
    \begin{proof}

   For simplicity, we will write out the proof in the one replica case and a  one particle cavity $e_{1,1}=1$. We will show the right hand side of the second line  \eqref{eq:impeq} can be shown to be a bounded continuous function of the overlap distribution. Once one has controls on the decay of the denominator, this is an exercise in applying the Stone-Weierstrauss theorem of approximation by polynomials. Namely, once the denominator is replaced by a polynomial, the resulting quantity is easily seen to be a function of the law of the infinite overlap matrix $R_{l,l'}$ of replicas of $\hat H_N$. By \cite{Panchenkobook}[Thm 2.14],  the law of the infinite overlap matrix is a function of the law of the overlap distribution.

    The limiting overlap distribution corresponding to $\hat H_N$ is $\zeta^*$, the same as that of $H_N$. This is a consequence of the fact that $\hat H_N$ and $H_N$ are generic models so the distribution of the overlap will be determined by the limiting value of the Free Energy \cite{Panchenkobook}[Thm 3.7]. The limiting value of the Free Energy of the two models can seen to be the same via interpolation.

    The formula at finite replica symmetry breaking is due to the fact that we can construct a Ruelle Probability Cascade with any given degree of finite replica symmetry breaking. The right hand side of \eqref{def:fform} is the appropriate quantity for the Ruelle Probability Cascade.

    \end{proof}
   We have the following lemma to bound the denominator. We introduce the notation $\td v^R_N$ to be the Gibbs' measure associated to $\hat H_N$.
    \begin{lm}
For any $k$, we have the following estimate.
\begin{equation}
P\left(\int_{S_N(N-R^2)} \int_{-C}^C e^{-\frac{s^2}{2} +Z(\hat \sigma) s + \left[\sqrt{R^2 +1} - \frac{s^2}{2\sqrt{R^2+1}}\right] Y(\hat \sigma)}   \td s \td \nu_N^R \le \frac{1}{L}\right) \le \frac{K(k)}{L^k}
\end{equation}
where $K(k)$ is some constant possibly depending on $k$
\end{lm}
\begin{proof}
We have the following
\begin{equation}
    \begin{aligned}
    &\int_{-C}^{C} \exp{\left\{-\frac{s^2}{2}+ s Z(\hat \sigma) + \left[\sqrt{R^2 +1} - \frac{s^2}{2\sqrt{R^2+1}}\right] Y(\hat \sigma)\right\}} \td s \\
    &\ge \int_{0}^{C}\exp{\left\{-\frac{s^2}{2} + \left[\sqrt{R^2 +1} - \frac{s^2}{2\sqrt{R^2+1}}\right] Y(\hat \sigma)\right\}} 2 \cosh{s Z(\hat \sigma)} \td s\\
    & \ge e^{-C^2/2} \int_{0}^{C} \exp\left\{ \left[\sqrt{R^2 +1} - \frac{s^2}{2\sqrt{R^2+1}}\right] Y(\hat \sigma)\right\}\td s
    \end{aligned}
\end{equation}
Thus, we have that
\begin{equation}
\begin{aligned}
&P\left(\int_{S_N(N-R^2)} \int_{-C}^C e^{-\frac{s^2}{2} +Z(\hat \sigma) s + \left[\sqrt{R^2 +1} - \frac{s^2}{2\sqrt{R^2+1}}\right] Y(\hat \sigma)}   \td s \td \nu^R_N \le \frac{1}{L}\right) \\
&\le P\left(\int_{S_N(N-R^2)} \int_0^C \exp\left\{\left[\sqrt{R^2 +1} - \frac{s^2}{2\sqrt{R^2+1}}\right] Y(\hat \sigma)\right\}\td s \td \nu^{R}_{N} \le \frac{e^{C^2/2}}{L}\right)
\end{aligned}
\end{equation}
Now we have that by Cauchy-Schwartz
\begin{equation}
\begin{aligned}
    &\left(\int_{S_N(N-R^2)}\int_0^C \exp\left\{\left[\sqrt{R^2 +1} - \frac{s^2}{2\sqrt{R^2+1}}\right] Y(\hat \sigma))\right\}\td s \td \nu^{R}_N\right)\\
    &\left(\int_{S_N(N-R^2)} \int_0^C  \exp\left\{-\left[\sqrt{R^2 +1} - \frac{s^2}{2\sqrt{R^2+1}}\right] Y(\hat \sigma)\right\}\td s \td \nu^{R}_N\right)\ge C^4
\end{aligned}
\end{equation}
Thus, we have that

\begin{equation}
\begin{aligned}
&P\left(\int_{S_N(N-R^2)} \int_0^C \exp\left\{\left[\sqrt{R^2 +1} - \frac{s^2}{2\sqrt{R^2+1}}\right] Y(\hat \sigma)\right\}\td s \td \nu^{R}_{N}  \le \frac{e^{C^2/2}}{L}\right)
\\&\le P\left(\int_{S_N(N-R^2)} \int_0^C \exp\left\{-\left[\sqrt{R^2 +1} - \frac{s^2}{2\sqrt{R^2+1}}\right] Y(\hat \sigma)\right\}\td s \td \nu^{R}_{N} \ge \frac{L C^4}{e^{C^2/2}}\right)
\end{aligned}
\end{equation}
Notice that $-Y(\hat \sigma)$ and $Y(\hat \sigma)$ have the same distribution. The will bound the right hand side of the above equation by Markov's inequality.

Let us find the $k$th moment; it is
\begin{equation}
    \mathbb{E}\int_{S_N(N-R^2)^k} \int_0^C\ldots\int_0^C \exp\left\{-\left[\sqrt{R^2 +1} - \frac{s_j^2}{2\sqrt{R^2+1}}\right] Y(\hat \sigma_j)\right\} \td s_1\ldots \td s_n (\td \nu^{R}_N)^{\otimes k}
\end{equation}
We can perform the integration in the $Y(\hat \sigma_j)$ Gaussian random variables first, noting that they are independent of the measure $\nnu$. Since $R$ and $s_i \in [0,C]$ are bounded, this is certainly some finite quantity. Thus, we can apply Markov's inequality in order to get the bound that
$$
P\left(\int_{S_N(N-R^2)} \int_0^C \exp\left\{-\left[\sqrt{R^2 +1} - \frac{s^2}{2\sqrt{R^2+1}}\right] Y(\hat \sigma)\right\}\td s \td \nu^{R}_{N} \ge \frac{L C^4}{e^{C^2/2}}\right) \le \frac{K(k)}{k^l}
$$
and we have proved the lemma.
\end{proof}

 \subsection{The Uniform limit in $R$}

 In this section, we will compute $\lim_{R \rightarrow \infty} F^{R,C,E}(\zeta)$ for those $\zeta$ that satisfy finite replica symmetry breaking. If we then show that if we can bound the difference $|F^{R,C,E}(\zeta) - \lim_{R \rightarrow \infty} F^{R,C,E}(\zeta)|$ uniformly along the sequence $\zeta_i$ approaching $\zeta^*$, we will be able to exchange the limit $\lim_{i \rightarrow \infty} \lim_{R\rightarrow \infty} = \lim_{R\rightarrow \infty} \lim_{i \rightarrow \infty}$

 We  can compute the limit $R \rightarrow \infty$ by applying the Bolthausen-Sznitman invariance for RPCs with respect to the Gaussian tilt $\prod_{i=1}^n Y^i(h_\alpha)\sqrt{R^2+1}$.  For simplicity of presentation, we will only present our results in the case that $n=1$.

\begin{lm}
Assume we are evaluating $F^{R,C,E}(\zeta)$ so that the overlap distribution $\zeta$ satisfies finite RSB. Let the support of the measure $\zeta$ be $0=q_0 \le q_1\le \ldots \le q_r=q_*$. We then have the following result

\begin{equation}\label{eq:limR}
\begin{aligned}
    \lim_{R \rightarrow \infty} &\mathbb{E} \prod_{l=1}^k \frac{\sum_{\alpha^l} w_{\alpha^l}  e^{\sqrt{R^2+1}Y\left(h_{\alpha^l}\right)}\prod_{i=1}^n\int_{-C}^C s_i^{e_{l,i}} e^{-\frac{1}{2}s_i^2[1+ \xi'(q_*)(1-q_*)] + Z^i\left(h_{\alpha^l}\right) s_i - Y^i\left(h_{\alpha^l}\right)\frac{s_i^2}{2\sqrt{R^2+1}}+\frac{(\xi'(1) - q_* \xi'(q_*))s_i^4}{4(R^2+1)} } \td s_i}{\sum_{\alpha^l}w_{\alpha^l} \prod_{i=1}^n e^{\sqrt{R^2+1} Y^i\left(h_{\alpha^l}\right)} \int_{-C}^C e^{-\frac{1}{2}s_i^2[1+ \xi'(q_*)(1-q_*)] + Z^i\left(h_{\alpha^l}\right) s_i - Y\left(h_{\alpha^l}\right) \frac{s_i^2}{2\sqrt{R^2+1}} +\frac{(\xi'(1) - q_* \xi'(q_*)s_i^4}{4(R^2+1)}  } \td s_i}=\\
    &\mathbb{E} \prod_{l=1}^k \frac{\sum_{\alpha_l} w_{\alpha_l}\prod_{i=1}^n  \int_{-C}^C s_i^{e_{l,i}} e^{-\frac{1}{2}s_i^2[1+ \hat{S}_\zeta ] + Z^i\left(h_{\alpha^l}\right) s_i } \td s_i}{\sum_{\alpha^l}w_{\alpha^l}  \prod_{i=1}^n \int_{-C}^C e^{-\frac{1}{2}s_i^2[1+ \hat{S}_\zeta] + Z^i\left(h_{\alpha^l}\right) s_i } \td s_i}
\end{aligned}
\end{equation}
where $\hat{S}_\zeta:=\xi'(q_*)(1-q_*) +\sum_{i=1}^{r}(q_r\xi'(q_r) - q_{r-1} \xi'(q_{r-1})) \zeta([0,q_{r-1}])$.

Moreover, the limit in $R$ is uniform in $\zeta$ for $\zeta_i$ a sequence of probability measures with finite RSB that approach $\zeta^*$.
\end{lm}
\begin{proof}
We will prove the result in the single replica, single cavity case with $e_{1,1}=1$. The proof in all other cases will be similar with marginally more involved computations.
Let $N$ denote the numerator of the expression in the first line of \eqref{eq:limR} under these conditions and let $D$ denote the denominator of said expression. We define $S$ to be $\sum_{\alpha} w_\alpha  e^{\sqrt{R^2+1}Y(h_\alpha)}$. Clearly we can rewrite the expression in the first line of \eqref{eq:limR} as $\mathbb{E} \frac{N S^{-1}}{D S^{-1}}$. Clearly, we will now be able to compute the quantity better if we understand the tilted Gibbs weight
\begin{equation}
    \frac{w_\alpha  e^{\sqrt{R^2+1} Y(h_\alpha)}}{\sum_{\alpha} w_\alpha  e^{\sqrt{R^2+1} Y(h_\alpha)}}
\end{equation}
Since we are just tilting by the exponential of independent Gaussians associated to each $\alpha$, at each level of the Ruelle probability cascade we  apply the standard Bolthausen-Sznitman invariance \cite{RPC} level by level which states that if $g_\alpha$ are i.i.d. $N(0,1)$ random variables independent of the  $u_\alpha$, which are the weights of a Poisson-Dirichlet process with parameter $m$,  then
\begin{equation}
    (e^{mt^2/2}u_\alpha e^{t g_\alpha}, g_\alpha - m t) \stackrel{d}{=} (u_\alpha, g_\alpha)
\end{equation}

If one applies the analogue of the above procedure for our RPC, we shift the Gaussian $Y(h_\alpha)$ by mean equal to $\sqrt{R^2+1} \sum_{i=1}^{r}(q_r\xi'(q_r) - q_{r-1} \xi'(q_{r-1})) \zeta([0,q_r])   $. Note that since we consider normalized weights, we cancel out the factor of $e^{mt^2/2}$ multiplying the weights of a Poisson-Dirichlet process.  One important remark to make is that even though this shift scales by $\sqrt{R^2+1}$ this multiplier is exactly canceled out by the $\frac{1}{\sqrt{R^2+1}}$ scaling factor associated to the only remaining term involving the $Y(h_\alpha)$. The net effect of this is to renormalize the variance of the Gaussian term associated to each spin. To simplify notation, recall $\hat S_\zeta:=\xi'(q_*)(1-q_*) +\sum_{i=1}^{r}(q_r\xi'(q_r) - q_{r-1} \xi'(q_{r-1})) \zeta([0,q_{r-1}]) $

As a result of applying said distributional equivalence, we see that the top line of \eqref{eq:limR} is given as
\begin{equation} \label{eq:b4rlimit}
  \lim_{R \rightarrow \infty} \mathbb{E} \frac{\sum_{\alpha} w_{\alpha}  \int_{-C}^C s e^{-\frac{1}{2}s^2[1+ \hat{S}_\zeta] + Z(h_\alpha) s - Y(h_\alpha)\frac{s^2}{2\sqrt{R^2+1}}+\frac{(\xi'(1) - q_* \xi'(q_*))s^4}{4(R^2+1)} } \td s}{\sum_{\alpha}w_{\alpha}  \int_{-C}^C e^{-\frac{1}{2}s^2[1+ \hat{S}_\zeta] + Z(h_\alpha) s- Y(h_\alpha)\frac{s^2}{2\sqrt{R^2+1}}+\frac{(\xi'(1) - q_* \xi'(q_*))s^4}{4(R^2+1)}  } \td s}
\end{equation}

We will show that the above limit is
\begin{equation}\label{eq:aftrrlimit}
    \mathbb{E}\left[ \frac{\sum_{\alpha} w_{\alpha}  \int_{-C}^C s e^{-\frac{1}{2}[1+\hat{S}_\zeta]s^2 +Z(h_\alpha) s} \td s}{\sum_{\alpha} w_{\alpha}  \int_{-C}^C  e^{-\frac{1}{2}[1+\hat{S}_\zeta]s^2+Z(h_\alpha) s} \td s}\right]
\end{equation}
and that the above limit can be taken uniformly in $\zeta$.

We denote the numerator of \eqref{eq:b4rlimit} as $X_R$ and the denominator as $Y_R$, while we denote the numerator of \eqref{eq:aftrrlimit} as $X_{\infty}$ and the denominator as $Y_{\infty}$.

We will bound the difference
\begin{equation}
    \mathbb{E}\left[\left|\frac{X_R}{Y_R} - \frac{X_\infty}{Y_\infty}\right|\right] \le \mathbb{E}\left[\left|\frac{X_R(Y_\infty - Y_R)}{Y_R Y_\infty}\right| + \left|\frac{X_R -X_\infty}{Y_\infty}\right|\right]
\end{equation}

We will bound $\mathbb{E}\left|\frac{X_R(Y_\infty - Y_R)}{Y_R Y_\infty}\right|$, the bound on the other part will be similar.

We first remark that $\frac{X_R}{Y_R}\le C$ as this is an upper bound on every individual ratio

$$
    \frac{  \int_{-C}^C s e^{-\frac{1}{2}[1+\hat{S}_\zeta]s^2 +Z^i(h_\alpha) s} \td s}{\int_{-C}^C  e^{-\frac{1}{2}[1+\hat{S}_\zeta]s^2+Z^i(h_\alpha) s} \td s}
$$
and each individual term in the denominator is positive.

It will now suffice to uniformly bound the quantity
$$
\mathbb{E}\left|\frac{Y_\infty -Y_R}{Y_\infty} \right|
$$

First notice that we are able to bound
$$
\begin{aligned}
&|Y_{\infty} - Y_R|=\\
 &\left|\int_{-C}^C e^{-\frac{1}{2}s^2[1+ \hat{S}_\zeta] + Z(h_\alpha) s- Y(h_\alpha)\frac{s^2}{2\sqrt{R^2+1}}+\frac{(\xi'(1) - q_* \xi'(q_*))s^4}{4(R^2+1)}  } \td s- \int_{-C}^C  e^{-\frac{1}{2}[1+\hat{S}_\zeta]s^2+Z(h_\alpha) s} \td s\right| \le\\
 & \left|\int_{-C}^C  e^{-\frac{1}{2}[1+\hat{S}_\zeta]s^2+Z(h_\alpha) s} \td s \right| \left[e^{|Y(h_\alpha)|\frac{C^2}{2\sqrt{R^2+1}}+ \frac{(\xi'(1) - q_* \xi'(q_*))C^4}{4(R^2+1)}}-1\right]
\end{aligned}
$$

Thus, it suffices to uniformly control
$$
\mathbb{E} \frac{\sum_{\alpha} w_{\alpha}  \int_{-C}^C  e^{-\frac{1}{2}[1+\hat{S}_\zeta]s^2+Z(h_\alpha) s} \td s \left[e^{|Y(h_\alpha)|\frac{C^2}{2\sqrt{R^2+1}}+ \frac{(\xi'(1) - q_* \xi'(q_*))C^4}{4(R^2+1)}}-1\right]}{\sum_{\alpha} w_{\alpha}  \int_{-C}^C  e^{-\frac{1}{2}[1+\hat{S}_\zeta]s^2+Z^i(h_\alpha) s} \td s}
$$

One can apply the Bolthausen-Sznitman invariance principle for RPC's on the random variable pair $(w_\alpha \int_{-C}^C  e^{-\frac{1}{2}[1+\hat{S}_\zeta]s^2+Z(h_\alpha) s} \td s, Z(h_\alpha))$ and, noting that $Y(h_\alpha)$ and $Z(h_\alpha)$ are independent of one another will give us that this is equal to
\begin{equation}\label{eq:redR}
\mathbb{E} \sum_\alpha w_\alpha\left[e^{|Y(h_\alpha)|\frac{C^2}{2\sqrt{R^2+1}}+ \frac{(\xi'(1) - q_* \xi'(q_*))C^4}{4(R^2+1)}}-1\right]
\end{equation}
for $w_\alpha$ the weights of an RPC with the same parameter.
This is bounded by the $L_1$ norm of the function $e^{|y|\frac{C^2}{2\sqrt{R^2+1}}+ \frac{(\xi'(1) - q \xi'(q))C^4}{4(R^2+1)}}-1$ where y is a Gaussian with variance $q \xi'(q)$. This clearly goes to 0 as R goes to $\infty$ and we have derived the infinite limit.
\end{proof}

To summarize what we have accomplished here, we have shown that
\begin{equation}
      \lim_{N \rightarrow \infty} \mathbb{E} \prod_{l=1}^{k}\langle(\sigma_1^l)^{e_{1,l}} \ldots (\sigma_n^l)^{e_{l,n}}\rangle= \lim_{C \rightarrow \infty} \lim_{R \rightarrow \infty} \lim_{i \rightarrow \infty}F^{R,C,E}(\zeta_i) = \lim_{C \rightarrow \infty} \lim_{i \rightarrow \infty}\lim_{R \rightarrow \infty} F^{R,C,E}(\zeta_i)
\end{equation}

Since each $\zeta_i$ satisfies finite RSB, the limit $F^{C,E}(\zeta_i)=\mathbb{E}\prod_{l=1}^k\left[ \frac{\sum_{\alpha_l} w_{\alpha_l} \prod_{i=1}^n \int_{-C}^C s_i^{e_{l,i}} e^{-\frac{1}{2}[1+\hat S_\zeta]s_i^2 +Z^i\left(h_{\alpha_l}\right) s_i} \td s_i}{\sum_{\alpha_l} w_{\alpha_l} \prod_{i=1}^n \int_{-C}^C  e^{-\frac{1}{2}[1+\hat S_\zeta]s_i^2+Z^i\left(h_{\alpha_l}\right) s_i} \td s_i}\right] $ is the value of $\lim_{R \rightarrow \infty} F^{R,C,E}(\zeta_i)$ where the $w_\alpha$ are the weights of some appropriate RPC.

 \subsection{The $C \rightarrow \infty$ limit}
 There is a natural guess for the $C \rightarrow \infty$ limit. It suffices to show that this limit exists and is uniform in the functional $\zeta$.

 \begin{lm} \label{lm:Cremoval}
 Assume we are evaluating $F^{C,E}(\zeta)$ so that the overlap distribution $\zeta$ satisfies finite RSB. Let the support of the measure $\zeta$ be $q_0 \le q_1\le \ldots \le q_r=q_*$. We then have the following result
 \begin{equation}
 \begin{aligned}
 &\lim_{C \rightarrow \infty} \mathbb{E}\left[\prod_{l=1}^{k} \frac{\sum_{\alpha_l} w_{\alpha_l} \prod_{i=1}^n \int_{-C}^C s_i^{e_{l,i}} e^{-\frac{1}{2}[1+\hat S_\zeta]s_i^2 +Z^i\left(h_{\alpha_l}\right) s_i} \td s_i}{\sum_{\alpha_l} w_{\alpha_l} \prod_{i=1}^n \int_{-C}^C  e^{-\frac{1}{2}[1+\hat S_\zeta]s_i^2+Z^i\left(h_{\alpha_l}\right) s_i} \td s_i}\right] =  \mathbb{E}\prod_{l=1}^k\left[ \frac{\sum_{\alpha_l} w_{\alpha_l} \prod_{i=1}^n \int_{-\infty}^\infty s_i^{e_{l,i}} e^{-\frac{1}{2}[1+\hat S_\zeta]s_i^2 +Z^i(h_\alpha) s_i} \td s_i}{\sum_{\alpha_l} w_{\alpha_l} \prod_{i=1}^n \int_{-\infty}^\infty  e^{-\frac{1}{2}[1+\hat S_\zeta]s_i^2+Z^i\left(h_{\alpha_l}\right) s_i} \td s_i}\right]=\\
 &=\mathbb{E}\prod_{l=1}^{k}\left[ \frac{\sum_{\alpha_l} w_{\alpha_l} \prod_{i=1}^n \rho^{Z^{i}\left(h_{\alpha_l}\right)}_{\zeta}(s^{e_{l,i}}) e^{\frac{(Z^i\left(h_{\alpha_l}\right))^2}{2[1+\hat S_\zeta]}}}{\sum_{\alpha_l} w_{\alpha_l} \prod_{i=1}^n  e^{\frac{\left(Z^i\left(h_{\alpha_l}\right)\right)^2}{2[1+\hat S_\zeta]}}}\right] = \mathbb{E}\prod_{l=1}^k[\sum_{\alpha_l} w_{\alpha_l} \prod_{i=1}^n \rho_{\zeta}^{(Z^i)'\left(h_{\alpha_l}\right)}(s^{e_{l,i}})]
 \end{aligned}
 \end{equation}
  where the construction of the random variables $(Z^i)'$ is as given in the appendix \ref{appendix} and the measure $\rho_{\zeta}^{.}$ is as in \eqref{eq:defrho}  .

 Moreover, the limit in $C$ can be taken uniformly in $\zeta$ for a sequence of finite RSB measures approaching $\zeta^*$.
 \end{lm}
 \begin{proof}
 We will show this computation in the case that $n=1,l=1$ and $e_{1,1}=1$, more general $E$ can be done using similar computations.
 Let $X_C$ denote the numerator
 $\sum_{\alpha} w_{\alpha} \int_{-C}^C s e^{-\frac{1}{2}[1+\hat{S}_\zeta]s^2 +Z(h_\alpha) s} \td s$,
 let $Y_C$ denote the denominator $\sum_{\alpha} w_{\alpha} \int_{-C}^C  e^{-\frac{1}{2}[1+\hat{S}_\zeta]s^2 +Z(h_\alpha) s} \td s$. Correspondingly, let $X_\infty$ be the numerator of the infinite limit $\sum_{\alpha} w_{\alpha} \int_{-\infty}^\infty s e^{-\frac{1}{2}[1+\hat{S}_\zeta]s^2 +Z(h_\alpha) s} \td s$ and $Y_\infty$ be the denominator of the infinite limit $\sum_{\alpha} w_{\alpha} \int_{-\infty}^\infty e^{-\frac{1}{2}[1+\hat{S}_\zeta]s^2 +Z(h_\alpha) s} \td s$.

 As before, we will bound
 $$
 \mathbb{E}\left|\frac{X_C}{Y_C} - \frac{X_\infty}{Y_\infty}\right| \le \mathbb{E} \left|\frac{X_C|Y_\infty - Y_C|}{Y_C Y_\infty}\right| + \left| \frac{ |X_\infty-X_C|}{Y_\infty}\right|
 $$

 We will only control the value of the former quantity. The methods in controlling the latter quantity would be roughly the same.

 We first bound $\frac{X_C}{Y_C}$ by C. We see that then it would suffice to show that $\mathbb{E}\left|\frac{Y_\infty - Y_C}{Y_\infty} \right|$ decays at a rate faster than $\frac{1}{C}$.
 Let $G_C(x)$ be the function $\frac{\int_{(-\infty,\infty) \setminus [-C,C]}e^{-\frac{1}{2}[1+\hat S_\zeta]s^2 +x s} \td s}{\int_{(-\infty,\infty)}e^{-\frac{1}{2}[1+\hat S_\zeta]s^2 +x s} \td s}$.
Also let $F(x)$ be the function $\int_{-\infty}^{\infty}e^{-\frac{1}{2}[1+\hat S_\zeta]s^2 +x s} \td s$.

We see that we can then write $\mathbb{E}\left|\frac{Y_\infty - Y_C}{Y_\infty} \right|$ as $\mathbb{E} \frac{\sum_{\alpha} w_\alpha G_C(Z(h_\alpha)) F(Z(h_\alpha))}{\sum_{\alpha}w_\alpha F(h_\alpha)}$. We would now apply the Bolthausen-Sznitman invariance principle to

$(w_\alpha F(Z(h_\alpha))$ to see that the expectation is equal to $\mathbb{E} \sum_{\alpha} w_\alpha G_C(Z'(h_\alpha)) = \mathbb{E} G_C(z')$ for some random variables $Z'(h_\alpha)$ and $z'$ is a random variable with the distribution of a single $Z'(h_\alpha)$. The random variables $Z'(h_\alpha)$ have been constructed in the Appendix \ref{appendix}.

One can observe that the function $G_C(x)$ is always less than 1 and ,furthermore, one can show that $G_C(x) \le e^{-C/4}$ when $|x| \le C/2$. Thus, we see that $\mathbb{E} G_C(z') \le \frac{B}{C^2}$ for some constant $B$ provided one can prove that $\mathbb{P}(z' \ge \frac{C}{2}) \le \frac{B}{C^2}$. By Lemma \ref{lm:2bound} we have that $\mathbb{E}[z'^2] \le K$ for some constant $K$ independent of the approximator, $\zeta_i$. Then, Markov's inequality shows that independently of $\zeta_i$, we have $\mathbb{P}(z' \ge \frac{C}{2}) \le \frac{4K}{C^2}$. This shows that the limit as $C \rightarrow \infty$ is uniform in the $\zeta_i$.

 \end{proof}

 To summarize again what we have established, we have
 \begin{equation}
      \lim_{N \rightarrow \infty} \mathbb{E} \prod_{l=1}^k \langle (\sigma_1^l)^{e_{l,1}} \ldots (\sigma_n^l)^{e_{l,n}}\rangle= \lim_{C \rightarrow \infty} \lim_{R \rightarrow \infty} \lim_{i \rightarrow \infty}F^{R,C,E}(\zeta_i) = \lim_{i \rightarrow \infty}\lim_{C \rightarrow \infty} \lim_{R \rightarrow \infty} F^{R,C,E}(\zeta_i)
\end{equation}
and that when $\zeta_i$ satisfies Finite Replica Symmetry Breaking, we have the explicit expression

$$\lim_{C \rightarrow \infty} \lim_{R \rightarrow \infty} F^{R,C,E}(\zeta_i) = \mathbb{E}\prod_{l=1}^k[\sum_{\alpha_l} w_{\alpha_l} \prod_{i=1}^n \rho_{\zeta}^{(Z^i)'\left(h_{\alpha_l}\right)}(s^{e_{l,i}})]$$ where the $w_\alpha$  and $(Z^i)'(h_\alpha)$ are calculated with respect to an RPC with parameters coming from $\zeta_i$. Our final step is to give a succinct representation as $i \rightarrow \infty$.

     \section{The local field representation of the spin distribution}

     Recall the notation from Theorem \ref{them:mainthm} and the preceding discussion. We have

     \begin{thm} \label{thm:mainpde}
     Consider a spherical spin glass $H_N$ whose limiting overlap distribution is given by $\zeta_*$ with $\sup\text{supp } \zeta_*= q_*$ and limiting Gibbs' measure $\nu$. Then we have that
     \begin{equation}
         \lim_{N\rightarrow \infty}\mathbb{E}[\prod_{l=1}^k \langle \prod_{i=1}^n (\sigma_i^l)^{e_{l,i}}\rangle_{H_N}] = \mathbb{E}[[\prod_{l=1}^k\prod_{i=1}^n \rho_{\zeta}^{[1+\hat S_\zeta]u_x(q_*, \mathcal{X}^i_{q_*}(\mathcal{\sigma}^l))}(s^{e_{l,i}})]
     \end{equation}
     where the second expectation is taken with respect to the randomness of the process $\mathcal{X}$ and $\mathcal{\sigma}^l$ are distributed i.i.d from $\nu^{\otimes \infty}$ as in Theorem \ref{them:mainthm}. We remark that if $\sigma^l$ are distributed according to $\nu^{\otimes \infty}$, then the law of the overlap distribution is $\zeta_*$
     \end{thm}

In the discussion that follows, we will restrict our computations to the case that we have a single cavity $i=1$ and each power $e_{l,i}=1$. The function $\rho_{\zeta}^{[1+\hat{S}_\zeta]u_x(q_*, \mathcal{X}^i_{q_*})(\mathcal{\sigma}^l)}(s)$ can be seen to be $u_x(q_*, \mathcal{X}^i_{q_*}(\mathcal{\sigma}^l))$

    \begin{lm}
    Consider a measure $\zeta$ satisfying finite replica symmetry breaking and consider $Z(h_\alpha)$ from \eqref{def:MagField}, the tilted function $Z'(h_\alpha)$ as defined in Appendix \ref{appendix} and the associated PDE quantities $\mathcal{Z}, \mathcal{X}$ from \eqref{def:cavfield} and \eqref{def:cavfield}.
    We have the following equality in distribution
    \begin{equation}
    \begin{aligned}
    &(Z(h_\alpha))_{\alpha} =_d (\mathcal{Z}_{q_*}(h_\alpha))_{\alpha}\\
    &(Z(h_\alpha)')_{\alpha} =_d (\mathcal{X}_{q_*}(h_\alpha))_{\alpha}
    \end{aligned}
    \end{equation}
    \end{lm}
    \begin{proof}
    The independent increments property of the Brownian motion will give us the first equality. We will now devote ourselves to proving the second equality. For simplicity of presentation, we will show the equivalence for a fixed $h_\alpha$. We will use Girsanov's theorem \cite{StVar} to determine the law of  $\mathcal{X}_t$ at $h_\alpha$. Let $Q$ be a measure space in which we have defined the Brownian motion \eqref{eq:BrownMot} . Consider a measure $P$ with Radon-Nikodym derivative
    \begin{equation}
        \frac{\td P}{\td Q} = e^{\int_0^t \zeta([0,s]) \td u(s,\mathcal{Z}_s) }
    \end{equation}
    Then under the measure $P$ the law of $\mathcal{Z}_t$ is distributed according to the law of $\mathcal{X}_t$. Considering a finite set of times $q_0,\ldots,q_r$ we have
    \begin{equation}
    \begin{aligned}
        \mathbb{E}_Q[F(\mathcal{X}_{q_0},\ldots,\mathcal{X}_{q_r})] &= \mathbb{E}_P[F(\mathcal{Z}_{q_0},\ldots, \mathcal{Z}_{q_r})] \\
        &= \int F(\mathcal{Z}_{q_0},\ldots,\mathcal{Z}_{q_r}) e^{ \int_0^t \zeta([0,s]) \td u(s,\mathcal{Z}_s)} \td Q(\mathcal{Z})\\
        &= \int F(\mathcal{Z}_{q_0},\ldots,\mathcal{Z}_{q_r}) \prod_{i=1}^r e^{\zeta([0,q_k])(u(q_k,\mathcal{Z}_{q_k}) - u(q_{k-1},\mathcal{Z}_{q_{k-1}}))} \td Q
    \end{aligned}
    \end{equation}
    Now, we use the fact that the Cole-Hopf solution  to the Parisi initial value problem is given by  is given by $u(q_k,\mathcal{Z}_{q_k}) = X_k(\mathcal{Z}_{q_1} - \mathcal{Z}_{q_0},\ldots,\mathcal{Z}_{q_k} - \mathcal{Z}_{q_{k-1}}) $ \cite{CheAuf1} and we have finished the proof.
    \end{proof}

    Now we cite a weak continuity result from \cite{JaAuSpin} regarding the behavior of the right hand side of \eqref{eq:pderep}

    Let $\mathcal{Q}_d$ denote the set of $d \times d$ ultrametric matrices of the form
    \begin{equation}
        \{(q_{ij})_{1\le i,j \le d}:q_{ij}=q_{ji}, q_{ij} \ge q_{ik} \wedge a_{kj} \forall i,j,k   \}
    \end{equation}
    Consider the space $\text{Pr}([0,1])$ equipped with the weak$*$ topology so the product space $Pr([0,1]) \times \mathcal{Q}_d$ is compact Polish. For any $Q \in \mathcal{Q}_d$, let $\{\sigma^i(Q): 1\le i \le d \} \in \mathcal{H}$, some Hilbert space, whose overlap matrix is $Q$. We can define the functional
    \begin{equation}
        \mathcal{R} (\zeta,Q) = \mathbb{E}[\prod_{i=1}^d u_x(q_*, \mathcal{X}_{q_*}(\sigma^i(Q)))]
    \end{equation}
    Then we have, this is similar to Lemma 3.3 in \cite{JaAuSpin}
    \begin{lm} \label{lm:weakcont}
    $\mathcal{R}$ is well-defined and if we let $\zeta_r$ be a collection of finite $RSB$ measures approaching $\zeta$ in the weak$*$ topology and $Q_r$ be a sequence of $d$x$d$ matrices approaching $Q$. Then we have that $\lim_{r \rightarrow \infty} \mathcal{R}(\zeta_r,Q_r) = \mathcal{R}(\zeta,Q)$
    \end{lm}
    \begin{proof}
    In the following discussion, we will use the notation in Lemma 3.3 in \cite{JaAuSpin}, since the law of the diffusion of the Local field process $\mathcal{X}$ only depends on the points $\sigma^i(Q)$ through the overlap distribution $Q$, the functional $\mathcal{R}(\zeta,Q)$ is well-defined.

    In order to show the continuity, it would suffice to show that the diffusion corresponding to the local field process $\mathcal{X}^{\zeta_r,Q_r}$ will approach the diffusion corresponding to the limit local field process $\mathcal{X}^{\zeta,Q}$. If we let $a^r_{ij}(t) = \mathbbm{1}(t \le q_{ij}^r)$, $b^r_{ij}(t,.) = \xi''\zeta_r u_x^r(t,.)$ be the coefficients corresponding to the diffusion $\mathcal{X}^{\zeta_r,Q_r}$, ideally one would like to show that we have uniformly that $a^r_{ij} \rightarrow a_{ij}$ and $b^r_{ij}(t,.) \rightarrow b_{ij}(t,.)$.

    However, we cannot show this uniform approach on the entire region $[0,1] \times \mathbb{R}$, rather we will only be able to show it for compact intervals. Our first reduction would be to show that it would only suffice to control the uniform approach of the diffusion coefficients on a compact interval. In order to make this rigorous, we will show the following steps.

    First, we take a cutoff and use that uniformly in the approximating sequence $\zeta_i$, we have that
    \begin{equation}
        |\mathbb{E}[\prod_{i=1}^d u^r_x(q_*, \mathcal{X}^{\zeta_r}_{q_*}(\sigma^i(Q_r)))] - \mathbb{E}[\prod_{i=1}^d \max(B,u^r_x(q_*, \mathcal{X}^{\zeta_r}_{q_*}(\sigma^i(Q_r))))]| \rightarrow 0
    \end{equation}
    as $B \rightarrow \infty$.
    This is a consequence of the moment bounds established in Lemma \ref{lm:2bound} and recalling that the process $\mathcal{X}^{\zeta}_{q_*}$ has the same distribution as $Z'$.

    Secondly, we use that  uniformly in $\zeta_i$, the local field process $\mathbb{P}(\exists(t,i):|\mathcal{X}^{\zeta_r}_t(\sigma^i(Q_r)| \ge C) \rightarrow 0$ as $C \rightarrow \infty$. This is a consequence of Lemma \ref{lem:submar} which shows that $|\mathcal{X}^{\zeta}_t|$ is a submartingale in $t$ (being equal to the process $|Z^{'}_k|$ for a measure satisfying finite RSB) and applying Doob's maximal inequality.

    Combining these two facts will allow us to see that if we instead consider the modified local field process $\tilde{\mathcal{X}}^{\zeta_r, Q_r}$ diffusion $\tilde{a}^r_{ij} := a^r_{ij}$ and $\tilde{b}^r_{ij}(t,.) := b^{r}_ij(t,.) \mathbbm{1}(x \le C)$ then we have that uniformly in the $\zeta_r$
    \begin{equation}
        \lim_{B\rightarrow \infty} \lim_{C \rightarrow \infty} \mathbb{E}[\prod_{i=1}^d \max(B,u^r_x(q_*, \tilde{\mathcal{X}}^{\zeta_r}_{q_*}(\sigma^i(Q_r))))] \rightarrow \mathbb{E}[\prod_{i=1}^d u^r_x(q_*, \mathcal{X}^{\zeta_r}_{q_*}(\sigma^i(Q_r)))]
    \end{equation}

    The benefit of this representation is that we have uniform convergence for the modified diffusion coefficients $\tilde{b}$ and $\tilde{a}$ by Lemma \ref{lem:unifconv}. Therefore, we can apply the Stroock-Varadhan, Theorem  11.1.4 in \cite{StVar} to get the convergence
    \begin{equation}
        \lim_{r \rightarrow \infty }\mathbb{E}[\prod_{i=1}^d \max(B,u^r_x(q_*, \tilde{\mathcal{X}}^{\zeta_r}_{q_*}(\sigma^i(Q_r))))] = \mathbb{E}[\prod_{i=1}^d \max(B,u^r_x(q_*, \tilde{\mathcal{X}}^{\zeta}_{q_*}(\sigma^i(Q)))]
    \end{equation}

    Since we have established earlier that the limits in $B$ and $C$ can all be taken uniformly in $\zeta_i$, we can exchange the $r \rightarrow \infty$ limit with the $\lim_{C\rightarrow \infty}$ and $\lim_{B \rightarrow \infty}$ limit to finally derive the result
    $\lim_{r \rightarrow \infty} \mathcal{R}(\zeta_r,Q_r) = \mathcal{R}(\zeta,Q)$.



    \end{proof}

    Now we finish the proof of our main result.
    \begin{proof}[Proof of \ref{thm:mainpde} ]
    We have already shown the result if we assume that $\zeta$ satisfies finite RSB. Using \ref{lm:weakcont} for our PDE functional and the fact that the spin distribution  $\lim_{N\rightarrow \infty}\mathbb{E}[\prod_{l =1}^k\langle \prod_{i =1}^{n} (\sigma_i^l)^{e_{i,l}}\rangle_{H_N}]$ is a weakly continuous functional of the overlap distribution gives us our desired conclusion by taking limits.
    \end{proof}

\appendix
\section{Computations for Finite RSB} \label{appendix}
\subsection{Construction of the Tilted Random Variable}

This section will contain various estimates that we use in order to understand the random variable $(Z')^{\zeta}$ and local field $\mathcal{X}^{\zeta}$ for measures $\zeta$ that satisfy finite replica symmetry breaking.

Let us first discuss the derivation of the random variables $Z'$. These random variables first appeared in the proof of Lemma \ref{lm:Cremoval}. Recall we wanted to understand the expression
\begin{equation} \label{eq:expectz}
    \mathbb{E}\left[\frac{\sum_{\alpha} w_{\alpha} Z(h_\alpha) e^{\frac{Z^2(h_\alpha)}{2[1+S_{\zeta}]}}}{\sum_{\alpha} w_{\alpha}  e^{\frac{Z^2(h_\alpha)}{2[1+S_{\zeta}]}}}\right]
\end{equation}
Recall that the variables $Z(h_\alpha)$ have a branching structure with $Z(h_\alpha) = \sum_{\beta \in p(\alpha)} z_{\beta}$ where $p(\alpha)$ designates the path between $\alpha$ and the root $\beta$. Each $z_\beta$ is an independent Gaussian random variable with variance $\xi'(q_{|\beta|}) - \xi'(q_{|\beta|-1})$.

We would now like to understand the expression \eqref{eq:expectz} by applying a modified Bolthausen-Sznitman invariance principle to the RPC variables $(w_\alpha e^{\frac{Z^2(h_\alpha)}{2[1+S_{\zeta}]}},Z(h_\alpha))$.

We would be able to show that there exists a variable $Z'(h_\alpha)$ with the following equality in distribution $(w_\alpha e^{\frac{Z^2(h_\alpha)}{2[1+S_{\zeta}]}},Z(h_\alpha)) =_d (w_\alpha, Z'(h_\alpha))$.

To define $Z'(h_\alpha)$ we would require the following functions.

We let
\begin{equation}
X_r(x_1,\ldots, x_r) = -\frac{(\sum_{i=1}^r x_i)^2}{2[1+S_\zeta]}
\end{equation}
and define the iteration
\begin{equation}
X_p(x_1,\ldots,x_p) = \frac{1}{\zeta([0,q_p])} \log \mathbb{E}_{p+1}[e^{\zeta([0,q_p]) X_{p+1}(x_1,\ldots,x_p,x_{p+1})}] \forall 1\le p \le r-1
\end{equation}
where $\mathbb{E}_{p+1}$ is integration with respect to the randomness of $x_{p+1}$ which is a  mean zero Gaussian random variable. In the notation of \cite{Panchenkobook}[Ch 4], this can alternatively be seen as integration with respect to the transition kernel $G_{p+1}(x_{p+1})= e^{-\frac{x_{p+1}^2}{2(\xi'(q_{p+1}) -\xi'(q_p) )}}$

We define the modified transition kernel
$$
G'_{p+1}(x',x'_{p+1}):= G_{p+1}(x'_{p+1}) e^{\zeta([0,q_p])(X_{p+1}(x',x'_{p+1}) -X_{p}(x'))}
$$

With this in hand, we are now able to describe the distribution of the new variables $Z'(h_\alpha)$: $Z'(h_\alpha)$ will decompose as a sum $\sum_{\beta \in p(\alpha)} z_\beta'$, where we generate the terms $z_\beta'$ via an iterative process.

Consider some node $\beta$ of the RPC and let $\beta_1,\ldots, \beta_{|\beta|-1}$ be the path from the root to $\beta$ where $|\beta|$ is the depth of the node $\beta$. We generate $z'_\beta$ using the values of the $z'_{\beta_i}$  and kernel $G'_{p+1}$ as in the following equation.
    \begin{equation}
        \mathbb{P}(z'_\beta \in A)  =\int_A G'_{|\beta|}(z'_{\beta_1},\ldots,z'_{\beta_{|\beta|-1}},z'_\beta) \td z'_\beta
    \end{equation}

    Using the equality in distribution $(w_\alpha e^{\frac{Z^2(h_\alpha)}{2[1+S_{\zeta}]}},Z(h_\alpha)) =_d (w_\alpha, Z'(h_\alpha))$, we can derive that
    \begin{equation} \label{eq:pderep}
        \mathbb{E}\left[ \frac{\sum_{\alpha} w_{\alpha} Z(h_\alpha) e^{\frac{(Z(h_\alpha))^2}{2[1+S_\zeta]}}}{\sum_{\alpha} w_{\alpha}   e^{\frac{(Z(h_\alpha))^2}{2[1+S_\zeta]}}}\right] = \mathbb{E}\left[\sum_{\alpha}w_{\alpha}  \frac{Z'(h_\alpha)}{2[1+S_\zeta]}\right]
    \end{equation}

    In this appendix, we compute various properties of the random variable $Z'(h_\alpha)$. In the computations that proceed, we will not be concerned regarding the various correlations of the $Z'(h_\alpha)$ relating to the position $h_\alpha$ along the RPC and, instead, only consider the distribution of a single $Z'$ with the same distribution as a single $Z'(h_\alpha)$. $Z'$ will decompose as a sum $z'_1+\ldots+z'_r$ where the $z'_k$ are distributed according to the Kernels $G'_k$.

    Our first lemma will involve an explicit computation of the functions $X_p$ for all p

    \begin{lm}\label{lm:recursion}
    \begin{equation}
        X_p(x_1,\ldots,x_p) = C_p + \frac{(\sum_{i=1}^p x_i)^2}{2[1+S_p]}
    \end{equation}
    with some constants $C_p$ and $S_p$ with the constant $S_p$ will satisfy the relationship $S_p = S_{p+1} - \zeta([0,q_p])(\xi'(q_{p+1}) - \xi'(q_p))$
    \end{lm}
    \begin{proof}
    We will prove this by induction with base case $p =r$ and going down.

We first assume the induction hypothesis for $X_{p+1}$ we have letting $M_p = \sum_{i=1}^p x_i$ and recalling the constant $1+ S_p = [1+S_{p+1}] - \zeta([0,q_p])(\xi'(q_{p+1}) - \xi'(q_p))$
\begin{equation}
    \begin{aligned}
        &X_p(x_1,\ldots,x_p) -C_p = \frac{1}{\zeta[0,q_p]} \log \int e^{\frac{\zeta[0,q_p]}{2[1+S_{p+1}]}(M_p^2 + 2 x_{p+1}M_p +x_{p+1}^2) } e^{-\frac{x_{p+1}^2}{2(\xi'(q_{p+1}) - \xi'(q_p))}} \td x_{p+1}\\
        &= \frac{1}{\zeta[0,q_p]} \log \int e^{-\frac{1+S_{p}}{2[1+S_{p+1}](\xi'(q_{p+1}) - \xi'(q_p)}(x_{p+1} - M_p \frac{\zeta([0,q_p])(\xi'(q_{p+1}) - \xi'(q_p))}{1+S_p})^2 + M_p^2[\frac{\zeta([0,q_p])^2 (\xi'(q_{p+1}) - \xi'(q_p))}{2(1+S_{p+1})(1+S_p)} + \frac{\zeta([0,q_p])}{2[1+S_{p+1}]}]} \td x_{p+1}
    \end{aligned}
\end{equation}
We can integrate the Gaussian above provided we have $1+S_p \ge 0$. We can also evaluate the constant as
\begin{equation}
    \frac{\zeta([0,q_p])^2 (\xi'(q_{p+1}) - \xi'(q_p))}{2(1+S_{p+1})(1+S_p)} + \frac{\zeta([0,q_p])}{2[1+S_{p+1}]} = \frac{\zeta([0,q_p])}{2 (1+S_p)}
\end{equation}
Using this we see that
\begin{equation}
    X_p(x_1,\ldots,x_p) = C_p + \frac{M_p^2}{2[1+S_p]}
\end{equation}

Notice that $S_p$ is increasing in $p$(namely, higher the $p$ the higher the value of $S_p$). The induction will be finished provided we have that $1+ S_0$ is positive.

We can express $1+S_0 = 1 + \sum_{i=0}^r \zeta(\{q_p\})(1-q_p)\xi'(q_p)$. This is manifestly positive.

    \end{proof}

    \begin{lm} \label{lem:submar}
        The process $|Z'_k| = |\sum_{j=1}^k z'_j|$  is a submartingale.
    \end{lm}
    \begin{proof}
    We first remark that the transition Kernel $G'_{k+1}$ will be a function that depends only on the values $z_1',z_2',\ldots z_k'$ through the sum $Z_k'$. This is due to \ref{lm:recursion} that shows that the function $X_k(x_1,\ldots,x_k)$ and $X_k(x_1,\ldots,x_{k+1})$ are functions only of the sum $(x_1+\ldots +x_k)^2$ and $(x_1+ \ldots+x_{k+1})^2$ respectively. We now have to perform the explicit computation which will allow us to understand the value of $z'_{k+1}$.

    We first remark that trivially when $Z_k$ is positive,  we have that
    $$
    \mathbb{E}[|Z'_{k+1}| | |Z_k|] \ge \mathbb{E}[|Z'_{k}| +z'_{k+1}| |Z'_k|]
    $$
    Notice that since the Kernel does not depend on the sign of $Z'_k$, we may as well assume that $Z'_k$ is positive when computing expectations.

    We will have proved $|Z'_k|$ is a submartingale when we establish
    $$
    \mathbb{E}[Z'_k + z'_{k+1}|Z'_k] > Z'_k
    $$
    when $Z'_k$ is positive.

    Let us now compute the necessary integral
    \begin{equation}\label{eq:1stmomentZ'}
        \mathbb{E}[z'_{k+1}|Z'_k] = \frac{\int z e^{\zeta([0,q_k]) \frac{(Z'_k +z)^2}{2[1+S_{k+1}]}} e^{-\frac{z^2}{2(\xi'(q_{k+1}) - \xi'(q_k))}} \td z}{\int  e^{\zeta([0,q_k]) \frac{(Z'_k +z)^2}{2[1+S_{k+1}]}} e^{-\frac{z^2}{2(\xi'(q_{k+1}) - \xi'(q_k))}} \td z} = Z_k' \zeta([0,q_k]) \frac{\xi'(q_{k+1}) - \xi'(q_k)}{1+S_k}
    \end{equation}
    Clearly, this will be positive as we have assumed $Z_k'$ to be positive and we have proved that the sequence $Z_j'$ is a submartingale.
    \end{proof}

    In the following lemmas, we want to study the behavior of the random variable $Z'$ as a function of the approximating measure $\zeta$. We will use the notation $Z'(\zeta)$ to specifically denote the random variable $Z'$ when it is derived using the base measure $\zeta$
    \begin{lm} \label{lm:2bound}
    $Z'(\zeta)$ is bounded uniformly for all $\zeta$.
    \end{lm}
    \begin{proof}
    This will be the consequence of computing moment bounds, which is an exercise in computing the expectation level by level. We will demonstrate the computation for the second moment; we can bound higher moments by the same method.
    What we will show by direct integration is that $\mathbb{E}[(Z'_{k+1})^2] = C_k \mathbb{E}[(Z'_k)^2] + B_k$ for some constants $C_k$ and $B_k$ that will occur naturally in the course of the computation.
    \begin{equation}
    \mathbb{E}[(Z'_{k+1}(\zeta))^2] = \mathbb{E}[(Z'_k(\zeta))^2] +2 \mathbb{E}[Z'_k(\zeta) \mathbb{E}[z'_{k+1}(\zeta)|Z'_k(\zeta)]] + \mathbb{E}[\mathbb{E}[(z'_{k+1}(\zeta))^2|Z'_k(\zeta)]]
    \end{equation}
    We have computed the first moment in the previous equation \eqref{eq:1stmomentZ'}. We now compute the second moment.

    \begin{equation}
    \begin{aligned}
        &\mathbb{E}[(z'_{k+1}(\zeta))^2|Z'_k(\zeta)] = \frac{\int z^2 e^{\zeta([0,q_k]) \frac{(Z'_k(\zeta) +z)^2}{2[1+S_{k+1}(\zeta)]}} e^{-\frac{z^2}{2(\xi'(q_{k+1}) - \xi'(q_k))}} \td z}{\int  e^{\zeta([0,q_k]) \frac{(Z'_k(\zeta) +z)^2}{2[1+S_{k+1}(\zeta)]}} e^{-\frac{z^2}{2(\xi'(q_{k+1}) - \xi'(q_k))}} \td z} \\
        &=  \frac{(1+S_{k+1})(\xi'(q_{k+1}) - \xi'(q_k))}{(1+S_k)}+ \left(Z_k' \zeta([0,q_k]) \frac{\xi'(q_{k+1}) - \xi'(q_k)}{1+S_k}\right)^2
    \end{aligned}
    \end{equation}
Similarly, we can derive
\begin{equation}
\begin{aligned}
     &\mathbb{E}[(z'_{k+1}(\zeta))|Z'_k(\zeta)] = \frac{\int z e^{\zeta([0,q_k]) \frac{(Z'_k(\zeta) +z)^2}{2[1+S_{k+1}(\zeta)]}} e^{-\frac{z^2}{2(\xi'(q_{k+1}) - \xi'(q_k))}} \td z}{\int  e^{\zeta([0,q_k]) \frac{(Z'_k(\zeta) +z)^2}{2[1+S_{k+1}(\zeta)]}} e^{-\frac{z^2}{2(\xi'(q_{k+1}) - \xi'(q_k))}} \td z} \\
        &=   2 \left(Z_k' \zeta([0,q_k]) \frac{\xi'(q_{k+1}) - \xi'(q_k)}{1+S_k}\right)
\end{aligned}
\end{equation}
In summary, we thus, have the inductive relation
\begin{equation}\label{Eq:indrel}
    \mathbb{E}[(Z'_{k+1}(\zeta))^2] = \left(\frac{1+S_{k+1}}{1 + S_k}\right)^2\mathbb{E}[Z_k'(\zeta)^2] + \frac{(1+S_{k+1})(\xi'(q_{k+1}) - \xi'(q_k))}{(1+S_k)}
\end{equation}

With this equation \eqref{Eq:indrel} hand, we can easily prove by induction the following inequality:
\begin{equation}
    \mathbb{E}[Z_k'(\zeta)^2] \le \prod_{i=1}^k (1+ (B+1)(\xi'(q_{i+1}) -\xi'(q_i)))^2
\end{equation}
where $B= \max(\xi'(1),1)$.
We have that $\frac{1 + S_{k+1}}{1+S_k} \le 1 + (\xi'(q_{k+1}) -\xi'(q_k))$

  As such we see
  \begin{equation}
  \begin{aligned}
       &\prod_{i=1}^{k+1} (1+ (B+1)(\xi'(q_{i+1}) -\xi'(q_i)))^2  -\left(\frac{1+S_{k+1}}{1 + S_k}\right)^2\mathbb{E}[Z_k'(\zeta)^2]\\
       & \ge  \prod_{i=1}^{k+1} (1+ (B+1)(\xi'(q_{i+1}) -\xi'(q_i)))^2 - (1+ (\xi'(q_{i+1}) -\xi'(q_i)))^2  \prod_{i=1}^k (1+ (B+1)(\xi'(q_{i+1}) -\xi'(q_i)))^2\\
       & \ge 2B (\xi'(q_{i+1}) -\xi'(q_i)) \ge \frac{(1+S_{k+1})(\xi'(q_{k+1}) - \xi'(q_k))}{(1+S_k)}
  \end{aligned}
  \end{equation}
  To get to the third line, we used the fact that $\prod_{i=1}^k(1+ (B+1)(\xi'(q_{i+1}) -\xi'(q_i)))^2\ge 1 $ and we can bound $\frac{1 +S_{k+1}}{1+S_k} \le 1+ \xi'(q_{k+1}) - \xi'(q_k) \le 2B $.

  We see that the expression $\prod_{i=1}^k(1+(B+1)(\xi'(q_{k+1}) -\xi'(q_k))) \le e^{(B+1) \sum_{i=1}^k \xi'(q_{k+1}) - \xi'(q_k))} \le e^{(B+1)^2}$
    \end{proof}
\subsection{The Parisi PDE}
The properties of the random variables $Z'$ that we have defined are closely related to the solutions of an associated Parisi PDE at approximating measure $\zeta$ with support at points $0=q_0 \le q_1 \le \ldots \le q_r = q^* \le 1$. We use the notation $q_{r+1} = 1$ where appropriate.
\begin{equation} \label{Eq:ParisiPDE}
\partial_t u^{\zeta}(t,x)= - \frac{\xi''(x)}{2} (\partial_x^2 u^{\zeta}(t,x) + \zeta([0,x])(\partial_x u^{\zeta}(t,x))^2)
\end{equation}
with initial data given by
\begin{equation}
    u(1,x) = \frac{x^2}{2[1+S_\zeta]}
\end{equation}
with $S_\zeta = \sum_{i=1}^{r+1} [q_i \xi'(q_i) - q_{i-1} \xi'(q_{i-1})] \zeta([0,q_{i-1}])$.

At $\zeta$ satisfying finite replica symmetry breaking, as we have assumed here, we have the Cole-Hopf solution of the Parisi PDE. Namely, for $t$ in between $q_j \le t <q_{j+1}$, we have
\begin{equation}
    u(x,t) = \frac{1}{\zeta([0,q_j])} \log\mathbb{E}[\exp{\zeta([0,q_j]u(x +z(\xi'(q_{j+1}) - \xi'(q_j)),q_{j+1})}]
\end{equation}
where $z$ is a standard Gaussian random variable. What we observe is that the function $u(x,q_j)$ is exactly the function $X_j(x_1,\ldots,x_j)$ for any set $x_1+\ldots+x_j =x$ from  lemma \ref{lm:recursion}. We have explicitly performed these Gaussian computations explicitly and we can quite readily talk about convergence properties of solutions to the Parisi PDE.

\begin{lm} \label{lem:unifconv}
Consider a sequence of measures $\zeta_i$ on $([0,1]$ satisfying finite RSB. Suppose the set $\zeta_i$ has some weak limit $\zeta$ in the weak$*$ topology. Then on any compact region $C$ of $\mathbb{R} \times [0,1]$ we have uniform convergence of the solution of the Parisi PDE  \eqref{Eq:ParisiPDE} to the limit solutions.
\begin{equation}
    u^{\zeta_i}(x,t) \rightarrow u^{\zeta}(x,t)
\end{equation}
uniformly in $C$.
\end{lm}


\begin{thebibliography}{99}
\bibliographystyle{amsplain}

\bibitem{SKMod}
D. Sherrington, S. Kirkpatrick, {\em A Solvable Model of Spin Glass} , Phys. Rev Letters(1975), No. 35, 1792-1796

\bibitem{Parisi1}
G. Parisi, {\em Infinite Number of Order Parameters for Spin-Glasses}, Phys. Rev Letters (1979), No. 43, 1754-1756

\bibitem{Parisi2}
G. Parisi, {\em A Sequence of Approximate Solutions to the SK Model for Spin Glasses}, J. Phys. A. NO. 13, L-115

\bibitem{Mezard1}
M. Mezard, G. Parisi, N. Sourlas, G. Toulouse, M.A. Virasoro, {\em On the Nature of the Spin Glass Phase}, Phys. Rev Letters(1984), No. 52, 1156

\bibitem{Mezard2}
M. Mezard, G. Parisi, N. Sourlas, G. Toulouse, M.A. Virasoro, {\em Replica Symmetry Breaking and the Nature of the Spin Glass Phase}, J. de Physique \textbf{45}, 843

\bibitem{Mezard3}
M. Mezard, G. Parisi,M.A. Virasoro, {\em Sping Glass Theory and Beyond}, World Scientific Lecture Notes in Physics, 9. World Scientific Publishing Co., Inc, Teaneck, NJ

\bibitem{DobSud}
L.N. Dobvysh, V.N. Sudakov, {\em Gram-de Finetti Matrices}, Zap. Nauchn. Sem. Leningrad Otdel. Mat. Inst. Stekov \textbf{119}(1982), 77-86

\bibitem{RPC}
E. Bolthausen, A Sznitman, {\em On Ruelle's Probability Cascade and an Abstract Cavity Method}, Comm. Math. Phys. \textbf{197}(1998),No. 2, 247-276

\bibitem{Panchenkobook}
D. Panchenko, {\em The Sherrington-Kirkpatrick Model}, Springer Monographs in Mathematics, Springer New York Heidelberg,2013

\bibitem{PanchenkoUltra}
D. Panchenko, {\em The Parisi Ultrametricity Conjecture}, Annals of Mathematics \textbf{177}(2013), Issue 1, 383-393

\bibitem{PanchenkoSpin}
D. Panchenko, {\em Spin Glass Models from the Point of View of Spin Distributions}, Annals of Probability \textbf{41}(2013), No. 3A, 1315-1361

\bibitem{JaAuSpin}
A. Jagannath, A. Auffinger, {\em On Spin Distributions for Generic p-spin Models}, J. Stat Phys \textbf{174}(2019), No. 316

\bibitem{Subag}
E. Subag, {\em The Geometry of the Gibbs' Measure of Pure Spherical Spin Glasses}, Invenciones Mathematicae \textbf{210}(2017), Issue 1, 135-209

\bibitem{SubagZei}
E. Subag, O. Zeitouni, {\em The Extremal Process of Critical Points of the Pure p-Spin Spherical Spin Glass Model}, Probability Theory and Related Fields \textbf{168}(2017),Issue 3-4, 773-820

\bibitem{CheAuf1}
W. Chen, A. Auffinger, {\em On Properties of Parisi Measures}, Probability Theory and Related Fields \textbf{161}(2015), No 3-4, 1429-1444

\bibitem{TalagrandParisi}
M. Talagrand, {\em The Parisi Formula}, Annals of Mathematics \textbf{163}(2006), Issue 1, 221-263

\bibitem{StVar}
D. Stroock, S.R.S. Varadhan, {\em Multidimensional Diffusion Processes}, volume 233(1979), Springer Science \& Business Media


\end{thebibliography}
\end{document}